\documentclass[11pt]{amsart}

\usepackage{amsmath}
\usepackage{amsxtra}
\usepackage{amscd}
\usepackage{amsthm}
\usepackage{amsfonts}
\usepackage{amssymb}
\usepackage{eucal}

\newcommand{\bra}[1]{\langle #1 |}        %bra
\newcommand{\ket}[1]{{| #1 \rangle}}      %ket
\newcommand{\nn}{\nonumber}
\newcommand{\bea}{\begin{eqnarray}}
\newcommand{\ena}{\end{eqnarray}}
\newcommand{\be}{\begin{eqnarray*}}
\newcommand{\en}{\end{eqnarray*}}
\def\bel{\begin{eqnarray}}
\def\enl{\end{eqnarray}}
\newcommand{\C}{{\mathbb C}}

\newcommand{\Z}{{\mathbb Z}}

\newcommand{\al}{{\alpha}}

\newcommand{\si}{{\sigma}}
\newcommand{\la}{{\lambda}}
\newcommand{\La}{{\Lambda}}
\newcommand{\mc}{\mathcal}

\def\P{\mathcal P}

\numberwithin{equation}{section}

\numberwithin{equation}{section}
\newtheorem{thm}{Theorem}[section]
\newtheorem{prop}[thm]{Proposition}
\newtheorem{lem}[thm]{Lemma}
\newtheorem{cor}[thm]{Corollary}
\theoremstyle{remark}
\newtheorem{rem}[thm]{Remark}
\newcommand{\gl}{\mathfrak{gl}}

\newcommand{\T}{\otimes}

\renewcommand{\deg}{\mathop{\rm deg}}

\newcommand{\one}{{\bf 1}}

\renewcommand{\triangle}{\Delta}

\newcommand{\W}{\mathcal{W}}

\newcommand{\F}{\mathcal{F}}

\newcommand{\sln}{\mathfrak{sl}_n}
\newcommand{\ba}{\bs a}
\newcommand{\bb}{\bs b}
\newcommand{\bla}{\mbox{$\bs \la$}}
\newcommand{\bxi}{\mbox{\boldmath$\xi$}}
\newcommand{\bet}{\mbox{\boldmath$\eta$}}
\newcommand{\bxis}{\mbox{\boldmath$\scriptstyle\xi$}}
\newcommand{\bets}{\mbox{\boldmath$\scriptstyle\eta$}}
\newcommand{\bN}{\bs N}
\newcommand{\M}{\mathcal{M}}
\newcommand{\E}{{\mathcal{E}}}
\newcommand{\chb}{\overline{\chi}}

\newcommand\Ref{\eqref}
\newcommand\bs{\boldsymbol}

\begin{document}
\title[Tensor products of Fock modules and $\mathcal{W}_n$ characters]
{Quantum continuous $\mathfrak{gl}_\infty$: 
Tensor products of Fock modules and $\mathcal{W}_n$ characters}

\author{B. Feigin, E. Feigin, M. Jimbo, T. Miwa and E. Mukhin}
\address{BF: Landau Institute for Theoretical Physics,
Russia, Chernogolovka, 142432, prosp. Akademika Semenova, 1a,   \newline
Higher School of Economics, Russia, Moscow, 101000,  Myasnitskaya ul., 20 and
\newline
Independent University of Moscow, Russia, Moscow, 119002,
Bol'shoi Vlas'evski per., 11}
\email{bfeigin@gmail.com}
\address{EF:
Tamm Department of Theoretical Physics, 
Lebedev Physics Institute, Russia, Moscow, 119991,
Leninski pr., 53 and\newline
French-Russian Poncelet Laboratory, Independent University of Moscow, 
Moscow, Russia }
\email{evgfeig@gmail.com}
\address{MJ: Department of Mathematics,
Rikkyo University, Toshima-ku, Tokyo 171-8501, Japan}
\email{jimbomm@rikkyo.ac.jp}
\address{TM: Department of Mathematics,
Graduate School of Science,
Kyoto University, Kyoto 606-8502,
Japan}\email{tmiwa@kje.biglobe.ne.jp}
\address{EM: Department of Mathematics,
Indiana University-Purdue University-Indianapolis,
402 N.Blackford St., LD 270,
Indianapolis, IN 46202}\email{mukhin@math.iupui.edu}
%\date{\today}

\begin{abstract}
We construct a family of irreducible representations of the quantum
continuous $\gl_\infty$ whose characters coincide with the characters
of representations in the minimal models of the $\mc W_n$ algebras of
$\gl_n$ type. In particular, we obtain a simple combinatorial model for all
representations of the $\mc W_n$-algebras appearing in the minimal models
in terms of $n$ interrelating partitions.
\end{abstract}

\maketitle

\section{Introduction}
In \cite{FFJMM} we initiated a study of representations of the algebra
which we denote by $\E$ and call quantum continuous $\gl_\infty$. This
algebra depends on two parameters $q_1,q_2$ and
is closely related to the Ding-Iohara algebra introduced in
\cite{DI} and then considered in \cite{FHHSY},
\cite{FT}. Conjecturally, the algebra $\E$ is isomorphic to the infinite
spherical double affine Hecke algebra constructed in \cite{SV1}, \cite{SV2}. 

We argued, that the representation theory of the algebra $\E$ has many
features similar to the representation theory of $\gl_\infty$. In
particular, it has a family of vector representations. The algebra
$\E$ is not a Hopf algebra. However there is ``a comultiplication
rule'' which under some conditions, defines an action of $\E$ on a
tensor product of $\E$-modules.  We used the comultiplication rule and
the vector representations to construct Fock modules by the standard
semi-infinite construction (which can also be called 
an inductive limit construction or semi-infinite wedge representation). 
Similarly to the case of $\gl_\infty$, the Fock modules have a
natural basis labeled by partitions.

On the other hand, due to the quantum nature of the algebra $\E$, the
representation theory of $\E$ is richer than that of
$\gl_\infty$. In \cite{FFJMM}, we used the semi-infinite construction
to define another family of modules for the case when
$q_1^{1-r}q_3^{k+1}=1$ with $r\in\Z_{>1}$, $k\in\Z_{>0}$. These
modules do not have $\gl_\infty$ analogs. Their bases are given by the
so-called $(k,r)$-admissible partitions $\la$ satisfying
$\la_i-\la_{i+k}\geq r$.

In \cite{FJMM1} it was shown that, when the parameters $q,t$ satisfy
$q^{1-r}t^{k+1}=1$, the Macdonald polynomials $P_\la(q,t)$ with
$(k,r)$-admissible partitions $\la$ form a basis of the smallest ideal
in the space of symmetric polynomials stable under the Macdonald
operators, see also \cite{K},\cite{SerVes}. The $\E$-module we
constructed is an inductive limit of these ideals in an appropriate
sense.

In this paper we continue the study of representations of $\E$ started
in \cite{FFJMM}.  We find that this algebra has a large class of
modules which are tame in the following sense. The algebra $\E$ has a
commutative subalgebra generated by $\psi_i^{+},\psi_{-i}^{-}$, where
$i\in \Z_{\geq 0}$, which serves as an analog of Cartan subalgebra. We
call a module tame if this subalgebra acts diagonally with the simple joint
spectrum.

In contrast to \cite{FFJMM}, to construct new modules, we do not
use the semi-infinite constructions but use the Fock spaces as
building blocks. Namely, we consider subquotients of the
 tensor product of several Fock spaces.

The Fock space depends on a continuous parameter $u$ and has a basis
labeled by partitions. Therefore the tensor product of $n$ Fock
modules depends on $n$ complex parameters $u_i$, where $i=1,\dots,n$, 
and has a basis
labeled by $n$-tuples of partitions. The subalgebra of $\psi^{\pm}_i$
acts diagonally in this basis. Remarkably, all submodules and
quotient modules we study also have a basis of $n$-tuples of partitions
with some conditions. It means that we never have to deal with linear
combinations of the basic vectors and all considerations are purely
combinatorial.

For generic parameters $q_1,q_2,u_i$, where $i=1,\dots,n$, 
the tensor products of Fock modules are irreducible. 
We consider special values which we call resonances.
In the case of the resonances, the $\E$ action on the tensor product of the
Fock spaces is not defined. However we find a ``subquotient''
for which it is well-defined. This subquotient is an irreducible $\E$-module
with a basis labeled by
$n$ partitions $\la^{(1)},\dots,\la^{(n)}$ with conditions
$$
\la_s^{(i)}\geq \la_{s+b_i}^{(i+1)}-a_i, \qquad {\rm where}\ i=1,\dots, n,\ s\in\Z_{>0}.
$$
Here $\la^{(n+1)}=\la^{(1)}$ and $a_i,b_i\in\Z_{\geq 0}$ are the parameters of the module.  The
conditions with $i=1,\dots,n-1$ correspond to the submodules appearing
under resonances $u_i/u_{i+1}=q_1^{a_i+1}q_3^{b_i+1}$, where
$q_3=(q_1q_2)^{-1}$. The $i=n$ condition corresponds to the quotient module
appearing under the resonance $q_1^{p'}q_3^p=1$, where
$p'=\sum_{i=1}^n(a_i+1)$, $p=\sum_{i=1}^n(b_i+1)$.

We use a recursion to compute the graded character of the subquotient
and find that it coincides with the character of the $\mathcal
W_n$-module from the minimal $(p',p)$ theory labeled by the $\widehat{\sln}$
weights $\bs\eta=\sum_{i=1}^{n}a_i\omega_i $ and 
$\bs\xi=\sum_{i=1}^{n}{b_i\omega_i}$
, see Theorem \ref{thm:Wn}.
The recursion in the case $n=2$ is similar to the recursion in \cite{A}.
 
 For the information on $\mc W_n$ and its representations 
we refer to \cite{FKW}. 
 The $\mc W_n$-modules are not well understood apart from the Virasoro 
 $n=2$ case and we hope that such a relatively simple combinatorial 
 description will shed
 new light on the subject. Moreover, there are many indications that
 the relation between $\E$ and $\mc W_n$ algebras is much deeper than
 the one described in this paper.

The paper is constructed as follows. In Section \ref{alg sec} we
recall the main definitions and constructions from \cite{FFJMM}. In
Section \ref{mod sec} we construct the subquotients of the tensor
products of Fock modules and prove that they are tame and irreducible.
In Section \ref{char sec} we describe the recursion for the set of
$n$-tuples of partitions with conditions,
give a solution of the recursion in a bosonic form and compare it to
the characters of $\mc W_n$ algebras. In Section \ref{iso sec} we show
several isomorphisms of $\E$-modules, in particular, the isomorphism
of the subquotients in the case $p=n+1$ to the $\E$-modules with bases
of $(k,r)$-admissible partitions of \cite{FFJMM}.

\section{Quantum continuous $\mathfrak{gl}_\infty$}\label{alg sec}
In this section we recall the basic definitions 
about quantum continuous $\mathfrak{gl}_\infty$
and its representations following \cite{FFJMM}.

\subsection{Algebra}
Let $q_1,q_2,q_3\in\C$ be  
complex parameters satisfying the relation $q_1q_2q_3=1$. 
In this paper we assume that neither of $q_i$ is a root of unity.
Let
\be
g(z,w)=(z-q_1w)(z-q_2w)(z-q_3w).
\en
The {\it quantum continuous $\mathfrak{gl}_\infty$} 
is an associative algebra $\E$ with generators
$e_i$, $f_i$, where $i\in\Z$, $\psi^+_i$, $\psi^-_{-i}$, where
$i\in\Z_{\geq 0}$,  and $(\psi_0^\pm)^{-1}$ satisfying the following 
defining relations:
\begin{gather}
\label{rel1}
g(z,w)e(z)e(w)=-g(w,z)e(w)e(z),\qquad
g(w,z)f(z)f(w)=-g(z,w)f(w)f(z),\\
\label{rel2}
g(z,w)\psi^{\pm}(z)e(w)=-g(w,z)e(w)\psi^{\pm}(z),\quad
g(w,z)\psi^{\pm}(z)f(w)=-g(z,w)f(w)\psi^{\pm}(z),\\
\label{rel3}
[e(z), f(w)]=\frac{\delta(z/w)}{g(1,1)}(\psi^{+}(z)-\psi^{-}(z)),\\
\label{rel4}
[\psi^\pm_i,\psi^\pm_j]=0, \qquad [\psi^\pm_i,\psi^\mp_j]=0, \\
\label{rel5}
\psi^\pm_0(\psi^\pm_0)^{-1}=(\psi^\pm_0)^{-1}\psi^\pm_0=1,\\
\label{rel6}
[e_0,[e_1,e_{-1}]]=0, \qquad [f_0,[f_1,f_{-1}]]=0.
\end{gather}
Here $\delta(z)=\sum_{n\in\Z} z^n$ denotes the formal
delta function and the generating series of the generators of $\E$ are given by
\be
e(z)=\sum_{i\in\Z} e_iz^{-i},
\quad  f(z)=\sum_{i\in\Z} f_iz^{-i},\quad
\psi^\pm(z)=\sum_{\pm i\ge 0}\psi^{\pm}_iz^{-i}.
\en
Note that the algebra $\E$ depends on the {\it unordered} set of parameters
$\{q_1,q_2,q_3\}$, as all $q_i$ enter the relations symmetrically
through the function $g(z,w)$.

The algebra generated by $\{\psi^+_{i},\psi_{-i}^-\}_{i\in\Z_{\geq 0}}$ 
is commutative. 
We call an $\E$-module $V$ {\it tame} if the operators 
$\{\psi^+_{i},\psi_{-i}^-\}_{i\in\Z_{\geq 0}}$ act by 
diagonalizable operators with simple joint spectrum.

The elements $\psi_0^\pm\in\E$ are central and invertible. 
For $x,y\in\C^\times$, we call an $\E$-module 
$V$  of {\it level $(x,y)$} if $\psi_0^+$ acts by 
$x$ and  $\psi_0^-$ by $y$. 

Let $\phi^\pm_{\emptyset}(z)\in\C[[z^{\mp1}]]$ be a formal series in
$z^{\mp1}$.  We call an $\E$-module $V$ {\it highest weight module
  with highest vector $v\in V$ and highest weight
  $\phi^\pm_{\emptyset}(z)$} if $v$ generates $V$ and 
\be f(z)v=0, \qquad
\psi^+(z)v=\phi^+_{\emptyset}(z) v, \qquad \psi^-(z)v=\phi^-_{\emptyset}(z) v.
\en

The comultiplication rule, see \cite{FFJMM}, \cite{DI}, is given by
\begin{gather}\label{tre}
\triangle e(z)= e(z)\T 1 + \psi^-(z)\T e(z),\\
\label{trf}
\triangle f(z)= f(z)\T \psi^+(z) + 1\T f(z),\\
\label{trpsi}
\triangle \psi^\pm(z)= \psi^\pm(z)\T \psi^\pm(z).
\end{gather}
These formulas do not define a comultiplication in the usual sense
since the right hand sides are not the elements of $\E\T\E$, as they
contain infinite sums.  However, under some conditions the
comultiplication rule does give an $\E$-module structure on $V_1\T
V_2$. We will discuss these conditions
somewhere else. In \cite{FFJMM} and in this paper we use the
comultiplication rule as a way to obtain the formulas for the $\E$
action and then we check directly that the answer is consistent with
the relations in the algebra $\E$.

%suppose given $\E$-modules $V_1$ and
%$V_2$ the operators $\triangle(x)$ on $V_1\T V_2$ are well-defined for
%all $x\in\E$. By that we mean that $\Delta (x)(v_1\otimes v_2)$, where
%$x\in\E$, $v_1\in V_1$, $v_2\in V_2$, is a linear combination of
%finitely many vectors and each coefficient is given by an absolutely
%convergent series of complex numbers.  Then the comultiplication rule
%does give an $\E$-module structure on $V_1\T V_2$, see \cite{FFJMM},
%\cite{DI}.  Thus, barring the well-definedness problems, one can use%
%the comultiplication rule as a comultiplication.

Note that this comultiplication rule is compatible with
permutations of $q_i$.

We say that a set of complex parameters $v_1,\dots,v_n$ is {\it generic} if no
monomial in $v_i$ is equal to $1$: $\prod_{i=1}^nv_i^{j_i}=1$ with
$j_i\in\Z$ implies that $j_1=\dots=j_n=0$.

Clearly, $q_1,q_2$ are generic if and only if $q_1,q_3$ are generic and
$q_1,q_2$ are generic if and only if $q_2,q_3$ are generic. 
If a set of parameters is generic then any subset of that set is also generic.
Throughout the paper we will assume that neither of $q_i$ is a root of
unity. That is we always assume that
$q_1$ is generic, $q_2$ is generic, $q_3$ is generic.

The algebra $\E$ is $\Z$-graded. 
The degrees of generators are given by
\bea\label{grading}
\deg e_i=1,\qquad \deg f_i=-1,\qquad 
\deg \psi^\pm_i=0.
\ena
In fact $\E$ is $\Z^2$-graded, see \cite{FFJMM}, Lemma 2.4,
but we do not need the second component of the grading in this paper.

\subsection{Vector representations}
For $u\in\C^\times$, let  $V(u)$ be a complex vector space 
spanned by basis $\ket{i}_u$, where $i\in\Z$. 

We also use the notation ${}_u\bra{i}$ with $i\in\Z$ for the dual
basis.  When it is clear we skip the subscript $u$ and write simply
$\ket{i}$, $\bra{i}$.

\begin{lem}\label{V_u}
The formulas
\begin{align*}
&(1-q_1)e(z)\ket{i}=\delta(q_1^iu/z)\ket{i+1}\,,
\\
&(q_1^{-1}-1)f(z)\ket{i}=\delta(q_1^{i-1}u/z)\ket{i-1}\,,
\\ &\psi^+(z)\ket{i}=\frac{(1-q_1^iq_3u/z)(1-q_1^iq_2u/z)}{(1-q_1^iu/z)(1-q_1^{i-1}u/z)}\ket{i},\\
&\psi^-(z)\ket{i}=\frac{(1-q_1^{-i}q_3^{-1}z/u)(1-q_1^{-i}q_2^{-1}z/u)}{(1-q_1^{-i}z/u)(1-q_1^{-i+1}z/u)}\ket{i},
\end{align*}
define a structure of an irreducible tame $\E$-module on $V(u)$ of level $(1,1)$. 
\end{lem}
\begin{proof}
These formulas define an $\E$-module by \cite{FFJMM}, \cite{FHHSY}. 
It is easy to check that it is tame and irreducible.
\end{proof}

We call the $\E$-module $V(u)$ the {\it vector representation}. Note that
$q_1$ plays a special role in the definition of $V(u)$ while $q_2$ and
$q_3$ participate symmetrically. Therefore there are two other vector
representations obtained from $V(u)$ by switching roles of $q_i$. We do
not use these other two modules in this paper.

Note that the vector representation is not a highest weight representation.

The comultiplication rule defines an $\E$-module structure on the
tensor product of vector representations $V(u_1)\otimes\dots\otimes
V(u_n)$ for generic $q_1,q_2,u_1,\dots,u_n$. This is again an irreducible
representation of level $(1,1)$.

If these parameters are not generic, the comultiplication rule
may fail to define the action of series $e(z)$ and $f(z)$. In general,
some matrix coefficients are well-defined and some are not. Also, some
matrix coefficients which are generically non-zero
become zero for special values of parameters.

The study of this phenomena is reduced to the case of $n=2$, and
the following simple lemma about tensor products of two vector
representations of $\E$ describes the situation.

\begin{lem}\label{2 point}
In $V(u)\otimes V(v)$, all matrix coefficients of operators $\psi^\pm(z)$, $e(z)$ and $f(z)$ are well-defined except possibly  for
\begin{eqnarray*}
\bra{i}\otimes\bra{j}\ e(z)\ \ket{i}\otimes\ket{j-1},\qquad
\bra{i}\otimes\bra{j}\ f(z)\ \ket{i+1}\otimes\ket{j}.
\end{eqnarray*}
Each of these matrix coefficients is undefined if and only if
$$
v/u=q_1^{i-j} \quad {\rm or} \quad  v/u=q_1^{i-j+1},
$$
and is equal to zero if and only if
$$
v/u=q_1^{i-j}q_2^{-1} \quad {\rm or} \quad v/u=q_1^{i-j}q_3^{-1}.
$$
Moreover the matrix coefficients 
\begin{eqnarray*}
\bra{i+1}\otimes\bra{j}\ e(z)\ \ket{i}\otimes\ket{j},\qquad
\bra{i}\otimes\bra{j}\ f(z)\ \ket{i}\otimes\ket{j+1}
\end{eqnarray*}
are always non-zero.
\end{lem}
\begin{proof}
The proof is straightforward.
\end{proof}

\subsection{Fock modules}
A {\it partition} $\la$ is the sequence $(\la_s)_{s\in\Z_{\geq 1}}$, 
such that $\la_s\in\Z$ and $\la_s\geq \la_{s+1}$ for all $s$. Let 
$\mc P$ be the set of all partitions. Let $\mc P^+\subset \mc P$ 
be the set of finite partitions:  $\la\in\mc P$ is in $\mc P^+$
if and only if only finitely many $\la_s$ 
are non-zero. 

For $s\in\Z_{\geq 1}$, denote $\bs 1_s=(\delta_{sm})_{m\in\Z_{\geq
    1}}$. For example, we have $\la+\bs
  1_s=(\la_1,\dots,\la_{s-1},\la_s+1,\la_{s+1},\dots)$.

For $u\in\C$, let  $\F(u)$ be a complex vector space 
spanned by  $\ket{\la}_u$, where $\la\in \mc P^+$. 

We use the notation  ${}_u\bra{\la}$ with $\la\in \mc P^+$ for the dual basis.
When it is clear we omit the subscript $u$ and write
simply $\bra{\la}$ and $\ket{\la}$.

For a partition $\la$ we denote by $\bra{\la}\psi^\pm(z)_i\ket{\la}$
the eigenvalue of the series $\psi^\pm(z)$ on the vector
$\ket{\la_i-i+1}_{uq_2^{-i+1}}\in V(uq_2^{-i+1})$, i.e.
$$
\psi^\pm(z)\ket{\la_i-i+1}_{uq_2^{-i+1}}=
\bra{\la}\psi^\pm(z)_i\ket{\la}\ket{\la_i-i+1}_{uq_2^{-i+1}}.
$$
The subscript $i$ in $\psi^\pm(z)_i$ indicates the $i$-th component $\la_i$ of the partition $\la$,
and at the same time the shifts
\begin{align*}
\la_i\mapsto\la_i-i+1,\quad
u\mapsto uq_2^{-i+1}
\end{align*}
in the vector $\ket{\la_i-i+1}_{uq_2^{-i+1}}$.

Similarly, we introduce the matrix coefficients $\bra{\la + {\bf
    1}_i}e(z)_i\ket{\la}$ and $\bra{\la}f(z)_i\ket{\la+{\bf 1}_i}$

\begin{align*}
\bra{\la + {\bf 1}_i}e(z)_i\ket{\la}=
\frac{\delta(q_1^{\la_i}q_3^{i-1}u/z)}{1-q_1},\qquad
\bra{\la}f(z)_i\ket{\la + {\bf 1}_i}=
\frac{q_1\delta(q_1^{\la_i}q_3^{i-1}u/z)}{1-q_1}.
\end{align*}
In what follows we multiply these delta functions by rational
functions $\bra{\la}\psi^\pm(z)_j\ket{\la}$.  We mean that they are
multiplied by the values of the rational functions at the support of
the delta function: $F(z)\delta(v/z)=F(v)\delta(v/z)$.

Let $\emptyset$ be the empty partition, $\emptyset_i=0$, $i\in\Z_{\geq 1}$.  
Introduce the series $\psi^{\pm}_{\emptyset}(z)\in\C[[z^{\pm 1}]]$ by
$$
\psi^+_{\emptyset}(z)=\frac{1-q_2z}{1-z},\qquad \psi^-_{\emptyset}(z)=
q_2\frac{1-q_2^{-1}z^{-1}}{1-z^{-1}}\ .
$$

Define the action of $e(z)$ on $\F(u)$ by 
\bea\label{e}
\bra{\la+{\bs 1}_i}e(z)\ket{\la} = \bra{\la + \bs 1_i}e(z)_i\ket{\la}
\ \prod_{j=1}^{i-1} {\bra{\la}\psi^+(z)_j\ket{\la}}
\ena
and setting all other matrix coefficients to be zero.
Define the action of $f(z)$ on  $\F(u)$ by 
\bea\label{f}
\bra{\la}f(z)\ket{\la +{\bf 1}_i} = 
\bra{\la}f(z)_i\ket{\la+\bs 1_i}\ \psi^+_{\emptyset}(uq_3^{i}/z)\
\prod_{j=i+1}^\infty \frac{\bra{\la}\psi^+(z)_j\ket{\la}}
{\bra{\emptyset}\psi^+(z)_j\ket{\emptyset}}
\ena
and setting all other matrix coefficients to be zero.
Define the action of $\psi^\pm(z)$ on  $\F(u)$ by
\bea\label{psi}
\bra{\la}\psi^\pm(z)\ket{\la}=\psi^\pm_{\emptyset}(u/z)\ \prod_{j=1}^\infty
\frac{\bra{\la}\psi^\pm(z)_j\ket{\la}}
{\bra{\emptyset}\psi^\pm(z)_j\ket{\emptyset}}
\ena
and setting all other matrix coefficients to be zero.

Note that although the formulas are written using 
infinite products, each product
in fact is finite since $\la_j=0=\emptyset_j$
 for all but finitely many indices $j$.

For example, explicitly we have
\begin{gather}\label{psi expl}
\psi^+(z)\ \ket{\la}
=\psi_{\la}(u/z)\ \ket{\la},\quad
\psi_{\la}(u/z)=
\frac{1-q_1^{\la_1-1}q_3^{-1}u/z}{1-q_1^{\la_1}u/z}
\prod_{i=1}^\infty
\frac{(1-q_1^{\la_i}q_3^iu/z)(1-q_1^{\la_{i+1}-1}q_3^{i-1}u/z)}{(1-q_1^{\la_{i+1}}q_3^iu/z)(1-q_1^{\la_i-1}q_3^{i-1}u/z)}.
\end{gather}

\begin{lem}
Formulas \Ref{e}, \Ref{f}, \Ref{psi} define a structure of an
irreducible tame $\E$-module on $\F(u)$ of level $(1,q_2)$.  It is 
a highest weight module 
with highest vector $\ket{\emptyset}$ and highest weight 
$\psi^\pm_{\emptyset}(u/z)$.
\end{lem}
\begin{proof}
These formulas define an $\E$-module by \cite{FFJMM}, cf. also \cite{FT}.
It is easy to check that it is tame and irreducible, see also Theorem \ref{tame thm 1} below. The highest weight conditions are obvious.
\end{proof}

We call $\F(u)$ the Fock module. It is an analog of a Fock module 
for $\gl_\infty$. This module appeared in \cite{FT} from 
geometric considerations.

\subsection{The module $G_{\bs a}^{k,r}$}
Assume $q_1^{1-r}q_3^{k+1}=1$, where $r,k+1\in\Z_{\geq 2}$.
More precisely, we mean that $q_1^xq_3^y=1$ if and only if
$x=(1-r)\kappa$, $y=(k+1)\kappa$ for some $\kappa\in\Z$.

Fix a sequence of non-negative integers $\bs a=(a_1,\dots,a_k)$ satisfying
$\sum_{i=1}^ka_i=r$,
and set $c_j=\sum_{i=1}^ja_i$.
Define the vacuum partition $\La^0\in \mc P$ by
\begin{align*}
\La^0_{\nu k+i+1}=-\nu r-c_i,\ {\rm where}\  
\nu\in\Z_{\geq 0},\ i=0,\dots, k-1.
\end{align*}
We define the sets of $(k,r)$-admissible partitions 
\begin{align*}
& \P_{\bs a}^{k,r}=
\{\La\in\mathcal P|\ \La_j-\La_{j+k}\ge r,\ {\rm for}\ j\in\Z_{\geq 1};
\quad\La_i=\La_i^0\ 
\hbox{\rm for all sufficiently large $i$}\}.
\end{align*}
We recall the semi-infinite construction of $\E$-module given in \cite{FFJMM}.
Let $W_{\bs a}^{k,r}(u)$ be the space spanned by  
$\ket\La$, where $\La\in \P^{k,r}_{\bs a}$.

\begin{rem}
Note that our notation is different from the one used in \cite{FFJMM}.
The module $W_{\bs c}^{k,r}(u)$ in \cite{FFJMM} is denoted by  
$W_{\bs a}^{k,r}(u)$ here, where $c_j=\sum_{i=1}^ja_i$. 
The formulas for the action are shorter in terms of $c_i$, on the other hand 
$a_i$ are natural parameters for many related objects appearing in this paper.
\end{rem}

Define series $\varphi^\pm_{\emptyset}(z)\in\C[[z^{\pm 1}]]$
by
$$
\varphi^+_{\emptyset}(z)=\frac{1-q_3z}{1-z},\qquad 
\varphi^-_{\emptyset}(z)=q_3\frac{1-q_3^{-1}z^{-1}}{1-z^{-1}}\ .
$$

Define the action of $e(z)$ on  $W_{\bs a}^{k,r}(u)$ by 
\bea\label{e fat}
\bra{\La+{\bf 1}_i}e(z)\ket{\La} = \bra{\La + \bs 1_i}e(z)_i
\ket{\La}\ \prod_{j=1}^{i-1} \bra{\La}\psi^-(z)_j\ket{\La}
\ena
and setting all other matrix coefficients to be zero.
Define the action of $f(z)$ on  $W_{\bs a}^{k,r}(u)$ by 
\begin{align}\label{f fat}
\bra{\La}f(z)\ket{\La +{\bf 1}_i} &\\
 =\bra{\La}& f(z)_i\ket{\La+\bs 1_i}
\prod_{j=i+1}^\infty \frac{\bra{\La}\psi^+(z)_j\ket{\La}}
{\bra{\La^0}\psi^+(z)_j\ket{\La^0}}
\prod_{j=0}^{k-1} \varphi^+_{\emptyset}
(uq_1^{\Lambda^0_{i+j+1}}q_3^{i+j}/z)) \notag
\end{align}
and setting all other matrix coefficients to be zero.
Define the action of $\psi^\pm(z)$ on $W_{\bs a}^{k,r}(u)$ by
\bea\label{psi fat}
\bra{\La}\psi^\pm(z)\ket{\La}=
\prod_{i=0}^{k-1} \varphi^\pm_{\emptyset}(uq_1^{-a_i}q_3^{i}/z))
\ \prod_{i\ge 1}\frac{\bra{\La}\psi^\pm(z)_i\ket{\La}}
{\bra{\Lambda^0}\psi^\pm(z)_i\ket{\Lambda^0}}
\ena
and setting all other matrix coefficients to be zero.

\begin{lem}\label{fat lem}
Suppose $q_1^{1-r}q_3^{k+1}=1$ with $r,k+1\in\Z_{\geq 2}$. Then
formulas \Ref{e fat}, \Ref{f fat}, \Ref{psi fat} define a structure of an
irreducible tame $\E$-module on $W_{\bs a}^{k,r}(u)$ 
of level $(1,q_3^k)$. 
It is a highest weight module 
with highest vector $\ket{\La^0}$ and highest weight 
$\prod_{i=0}^{k-1} \varphi^\pm_\emptyset(uq_1^{-a_i}q_3^{i}/z)$.
\end{lem}
\begin{proof} These formulas define an $\E$-module by \cite{FFJMM}. 
  It is easy to check that it is tame and irreducible. (It also
  follows from Theorem \ref{iso thm} of this paper.) The highest
  weight conditions are obvious.
\end{proof}

Let $q_1^{p'}q_3^p=1$ and $p=k+1$, $p'=k+r$, where as above
$r,k+1\in\Z_{\geq 2}$. This is equivalent to $q_1^{1-r}q_2^{k+1}=1$.
Again, by that we mean that $q_1^xq_3^y=1$ is and only if
$x=(k+r)\kappa$, $y=(k+1)\kappa$ for some $\kappa\in\Z$.
  
  Let $G_{\bs a}^{k,r}$ be the space spanned by $\ket\la$, where
  $\la\in \P^{k,r}_{\bs a}$. Define the action of operators
  $e(z),f(z), \psi^\pm(z)$ on $G_{\bs a}^{k,r}$ by formulas \Ref{e
    fat}, \Ref{f fat}, \Ref{psi fat}, where $q_2$ is replaced with
  $q_3$, and $q_3$ is replaced with $q_2$.

\begin{lem}\label{G lem}
Suppose  $q_1^{p'}q_3^p=1$ and $p=k+1$, $p'=k+r$, where $k+1,r\in\Z_{\geq 2}$. 
Then formulas \Ref{e fat}, \Ref{f fat}, \Ref{psi fat},
where $q_2$ is replaced with $q_3$, and $q_3$ is replaced with $q_2$,
define a structure of an
irreducible tame $\E$-module on $G_{\bs a}^{k,r}$
of level $(1,q_2^k)$. 
It is a highest weight module 
with highest vector $\ket{\La^0}$ and highest weight 
$\prod_{i=0}^{k-1} \psi^\pm_\emptyset(uq_1^{-a_i}q_2^{i}/z))$.
\end{lem}
\begin{proof}
The lemma follows from Lemma \ref{fat lem} and the symmetry of the algebra $\E$
with respect to permutations of parameters $q_1,q_2,q_3$.
\end{proof}

%%%%%%%%%%%%%%%%%%%%%%%%%%%%%%%%%%%%%%%%%%%

\section{Construction of $\E$-modules}\label{mod sec}
In this section we construct $\E$-modules $\M_{\bs a,\bs b}^{p',p}(u)$
as subquotients of tensor products of Fock modules.

\subsection{Generic tensor products}
Consider a tensor product of $n$ Fock modules
$\F(u_1)\otimes\dots \otimes\F(u_n)$ with $n\geq 2$. 
In this section we assume that $q_1,q_2,u_1,\dots,u_n$ are generic.

A basis of $\F(u_1)\otimes\dots \otimes\F(u_n)$ is given by
$\ket{\la^{(1)}}_{u_1}\otimes\dots\otimes\ket{\la^{(n)}}_{u_n}$, where
$\la^{(i)}\in\mc P$ for $i=1,\dots,n$.

We use the following notation. We write
basic vectors in $\F(u_i)$ with upper index $i$ and skip the index
$u_i$: we write simply $\ket{\la^{(i)}}$ instead of
$\ket{\la^{(i)}}_{u_i}$.  Moreover, we use the notation
$\ket{\la^{(1)},\dots,\la^{(n)}}$ and $\bra{\la^{(1)},\dots,\la^{(n)}}$ for
$\ket{\la^{(1)}}\otimes\dots\otimes\ket{\la^{(n)}}$ and
$\bra{\la^{(1)}}\otimes\dots\otimes\bra{\la^{(n)}}$.
Sometimes we use the bold font notation 
$\bs \la=(\la^{(1)},\dots,\la^{(n)})$ and then we denote $\bs 1_s$ in the $i$-th place by $\bs 1^{(i)}_s$:  
$$
\bs\la+\bs 1_s^{(i)}=(\la^{(1)},\dots,\la^{(i)}+\bs 1_s,\dots, \la^{(n)}).
$$

Define a $\Z$ grading on $\F(u_1)\otimes\dots\otimes\F(u_n)$ 
by (cf. \Ref{grading}):
$$
\deg( \ket{\la^{(1)}}\otimes\dots\otimes\ket{\la^{(n)}})=
\sum_{i=1}^n|\la^{(i)}|=\sum_{i=1}^n\sum_{s=1}^\infty \la^{(i)}_s.
$$

\begin{lem} Assume that $q_1,q_2,u_1,\dots,u_n$ are generic.
The comultiplication rule defines on 
$\F(u_1)\otimes\dots \otimes\F(u_n)$ a structure 
of an irreducible graded tame $\E$-module of level $(1,q_2^n)$.

The module $\F(u_1)\otimes\dots \otimes\F(u_n)$
is a highest weight module 
with highest vector 
$\ket{\emptyset^{(1)}}\otimes\dots\otimes\ket{\emptyset^{(n)}}$ 
and highest weight $\prod_{i=1}^n\psi^\pm_{\emptyset}(u_i/z)$.
\end{lem}
\begin{proof}
To check that the comultiplication rule gives well-defined formulas, 
it is sufficient to consider the case $n=2$.  
In the case $n=2$, the well-definedness is obvious due to Lemma \ref{2 point}.

These formulas give a well-defined action of $\E$, that is that the
relations in $\E$ are respected. Indeed, the check is reduced to the
case of a tensor product of vector representation, which it is done in
Lemma 2.5 in \cite{FFJMM}.

Let us prove the simplicity of the spectrum of $\psi(z)$. 
Recall the eigenvalue $\psi_\la(u/z)$ of $\psi^+(z)$ 
on $\ket{\la}$ in $\F(u)$, see \Ref{psi expl}. Assume that 
\be
\prod_{i=1}^n \psi_{\la^{(i)}}(u_i/z)=\prod_{i=1}^n \psi_{\mu^{(i)}}(u_i/z).
\en
We need to show that this implies $\la^{(i)}=\mu^{(i)}$ for $i=1,\dots,n$.

Note that $\psi_{\la^{(i)}}(u_i/z)$ has a pole at $z=q_1^{\la^{(i)}_1}u_i$. 
Since $q_1,q_2,u_1,\dots,u_n$ 
are generic this pole can be canceled only by the poles
$z=q_1^{\mu^{(i)}_1}u_i$ or $z=q_1^{\mu^{(i)}_1-1}u_i$. 
The latter is impossible because in such a case $\la^{(i)}_1=\mu^{(i)}_1-1$, 
and
the pole  $z=q_1^{\mu^{(i)}_1}u_i$ is not canceled. Therefore we obtain
 $\la^{(i)}_1=\mu^{(i)}_1$ for $i=1,\dots, n$. Cancel the terms with 
$\la^{(i)}_1$ and $\mu^{(i)}_1$, and replace $u_i$  with $u_i/q_3$. 
Then the same argument gives  $\la^{(i)}_2=\mu^{(i)}_2$ for $i=1,\dots, n$. 
Repeating the process we obtain that $\la^{(i)}=\mu^{(i)}$ for $i=1,\dots, n$.

Since the representation is tame, to show it is irreducible, it is sufficient
to check that the matrix coefficients
$\bra{\bs\la+\bs 1_s^{(i)}}e(z)\ket{\bs \la}$ and
$\bra{\bs\la}f(z)\ket{\bs\la+\bs 1_s^{(i)}}$ are non-zero for all $i,s$. 
This is reduced to $n=2$ case where it follows from Lemma \ref{2 point}.
\end{proof}

\subsection{Resonance in $u_i/u_{i+1},q_1,q_3$}
We turn to special values of parameters. In doing so we always keep in mind
that the matrix coefficients of the considered modules 
are rational functions of
parameters, sometimes multiplied by the delta functions. When we go to
special values of parameters we just take limits of these rational
functions. In particular, we first cancel factors at generic values of
parameters as much as possible and then simply substitute the special
values.

Consider the tensor product of $n$ Fock modules
$\F(u_1)\otimes\dots \otimes\F(u_n)$ with $n\geq 2$ and let 
\bea\label{shift n}
u_i=u_{i+1}\ q_1^{a_i+1}q_3^{b_i+1}, \ {\rm where}\ a_i,b_i\in\Z_{\geq 0}\ {\rm and}\ i=1,\dots,n-1.
\ena
Let $\bs a=(a_1,\dots,a_{n-1})$, $\bs b=(b_1,\dots,b_{n-1})$, $u_1=u$ and let
\be
\M_{\bs a,\bs b}(u)=\operatorname{span}\{ \ket{\la^{(1)},\dots,\la^{(n)}}\ |\ \la^{(i)}_s\geq \la^{(i+1)}_{s+b_i}-a_i,\ {\rm where}\ s\in\Z_{\geq 1},\ i=1,\dots,n-1\} .
\en
Note that if $a_i$ was negative for some $i$ then the space $\M_{\bs a,\bs b}$ 
would be trivial.

The following lemma shows that the definition of $\M_{\bs a,\bs b}(u)$ 
is in fact a superposition of $n=2$ conditions. Note that for $1\leq i<j\leq n$ we have
$$
u_i=u_jq_1^{a_{ij}+1}q_3^{b_{ij}+1}, \qquad
a_{ij}=\sum_{l=i}^{j-1} (a_l+1)-1,\quad b_{ij}=\sum_{l=i}^{j-1} (b_l+1)-1. 
$$

\begin{lem}\label{2n submodule} We have
$\ket{\la^{(1)},\dots,\la^{(n)}}\in
\M_{\bs a,\bs b}(u)$ if and only if for all $i,j$, $1\leq i<j\leq n$, 
$\ket{\la^{(i)}}\otimes\ket{\la^{(j)}}\in \M_{a_{ij},b_{ij}}(u_i)$.
\end{lem}
\begin{proof}
The lemma is straightforward.
\end{proof}

We have an obvious inclusion of vector spaces 
$\M_{\bs a,\bs b}(u)\to \F(u_1)\otimes\dots \otimes\F(u_n)$. 
In particular, the space $\M_{\bs a,\bs b}(u)$ inherits the $\Z$-grading.

We define the action of operators $\psi^\pm(z), e(z),f(z)$ on $\M_{\bs
  a,\bs b}(u)$ using the action of $\E$ on the tensor product
$\F(u_1)\otimes\dots\otimes
\F(u_n)$. Namely, let the matrix coefficients of operators
$\psi^\pm(z), e(z),f(z)$ acting on $\M_{\bs a,\bs b}(u)$ in the basis
$\ket{\la^{(1)},\dots,\la^{(n)}}$ be the same as the corresponding
matrix coefficients for the tensor action.

\begin{prop}\label{submodule well-defined n}
Assume that $q_1,q_2,u$ are generic. Then the action of operators
$\psi^\pm(z), e(z),f(z)$ in
$\M_{\bs a,\bs b}(u)$ is well-defined and gives a structure of a graded
$\E$-module.
\end{prop}
\begin{proof} 
Consider the case $n=2$. Let $a_1=a, b_1=b$.

It is sufficient to perform the following checks. 
\begin{enumerate}
\item 
If $\ket{\la,\mu}\in \M_{a,b}(u)$, then the matrix coefficients
$
\bra{\la,\mu}\ e(z)\ \ket{\la,\mu-\one_s}
$, 
$
\bra{\la,\mu}\ f(z)\ \ket{\la+\one_s, \mu}
$
are well defined. 
\item If $\ket{\la,\mu}\in \M_{a,b}(u)$, then 
\begin{align*}
&\ket{\la,\mu+\one_s}\not\in \M_{a,b}(u)\Rightarrow 
\bra{\la,\mu+\one_s}\ e(z)\ \ket{\la,\mu}=0,
\\
&\ket{\la-\one_s,\mu}\not\in \M_{a,b}(u)
\Rightarrow 
\bra{\la-\one_s,\mu}\ f(z)\ \ket{\la,\mu}=0\,.
\end{align*}
\end{enumerate}

All the checks are straightforward using Lemma \ref{2 point}. 

For
example, let $\ket{\la}\otimes\ket{\mu}, \ket{\la+\bs
  1_s}\otimes\ket{\mu}\in \M_{a,b}(u)$ and consider
$\bra{\la}\otimes\bra{\mu}\ f(z)\ \ket{\la+\bs 1_s}\otimes\ket{\mu}$. By
Lemma \ref{2 point} the poles of this matrix coefficient happen if
$$
uq_2^{1-s}/(vq_2^{1-l})=q_1^{\mu_l-l-\la_s+s}\quad {\rm or}\quad
uq_2^{1-s}/(vq_2^{1-l})=q_1^{\mu_l-l-\la_s+s-1}.
$$ 
Which means $q_1^{\mu_l-\la_s-a-1}=q_3^{s-l+b+1}$ or 
$q_1^{\mu_l-\la_s-a-2}=q_3^{s-l+b+1}$. Equivalently, $l=s+b+1$, and
$\la_s=\mu_{s+b+1}-a-1$ or $\la_s=\mu_{s+b+1}-a-2$. 
This is impossible, because 
$$
\la_s\geq \mu_{s+b}-a\geq \mu_{s+b+1}-a>\mu_{s+b+1}-a-1>\mu_{s+b+1}-a-2.
$$

Similarly, let $\ket{\la}\otimes\ket{\mu}\in \M_{a,b}(u)$ and
$\ket{\la}\otimes\ket{\mu+\one_s}\not\in \M_{a,b}(u)$. 
Then we have $s-b\geq 1$ and $\la_{s-b}=\mu_{s}-a$. 
It follows that the coefficient
$\bra{\la}\otimes\bra{\mu+\one_s}\ e(z)\ \ket{\la}\otimes\ket{\mu}$ 
vanishes by Lemma \ref{2 point}, since
$vq_2^{1-s}/(uq_2^{1-(s-b)})=q_1^{b-a}q_2=q_1^{\la_{s-b}-\mu_s+b-1}q_3^{-1}$.

We omit further details.

Since all the necessary checks reduce to the case of $n=2$ due to
Lemma \ref{2n submodule}, the general case 
Proposition \ref{submodule well-defined n} follows.
\end{proof}

\begin{thm}\label{tame thm 1}
Assume that $q_1,q_2,u$ are generic.
Then $\M_{\bs a,\bs b}(u)$ is an irreducible, tame, highest weight 
$\E$-module with highest weight 
$\prod_{i=1}^n\psi^\pm_{\emptyset}(u_i/z)$.
\end{thm}
\begin{proof}
First, let us show the module is tame.
Assume that for some $\ket{\la^{(1)},\dots,\la^{(n)}}, 
\ket{\mu^{(1)},\dots,\mu^{(n)}}\in \M_{\bs a,\bs b}(u)$, we have
\be
\prod_{i=1}^n \psi_{\la^{(i)}}(u_i/z)=\prod_{i=1}^n \psi_{\mu^{(i)}}(u_i/z).
\en
Recall that $\psi_{\la^{(i)}}(u_i/z)$ has a pole at $z=q_1^{\la^{(i)}_1}u_i$, see \Ref{psi expl}. 
We show that for $i=1,\dots,n$, $\la^{(i)}_1=\mu^{(i)}_1$
by showing that the  pole $z=q_1^{\la^{(i)}_1}u_i$ in the left hand side
is canceled by the pole  $z=q_1^{\mu^{(i)}_1}u_i$ in the right hand side. 

Suppose the pole $z=q_1^{\la^{(i)}_1}u_i$ is canceled by zeroes of
$\psi_{\la^{(j)}}(u_j/z)$.  There are two possible cases:
$q_1^{\la^{(i)}_1}u_i=q_1^{\la^{(j)}_s}q_3^{s}u_j$, where
$s\in\Z_{\geq 1}$ and
$q_1^{\la^{(i)}_1}u_i=q_1^{\la^{(j)}_{s+1}-1}q_3^{s-1}u_j$, where
$s\in\Z_{\geq 0}$.

Suppose 
$q_1^{\la^{(i)}_1}u_i=q_1^{\la^{(j)}_s}q_3^{s}u_j$, where $s\in\Z_{\geq 1}$. 
Since $s\geq 1$, 
and \Ref{shift n} holds, it implies that $j>i$ and  
\begin{eqnarray*}
\la^{(i)}_1=\la^{(j)}_s-\sum_{l=i}^{j-1}(a_l+1), \qquad
s=\sum_{l=i}^{j-1}(b_l+1).
\end{eqnarray*}
But then we have
$$
\la^{(i)}_1\geq \la^{(j)}_{1+\sum_{l=i}^{j-1}b_l}-\sum_{l=i}^{j-1}a_l>
\la^{(j)}_{\sum_{l=i}^{j-1}(b_l+1)}-\sum_{l=i}^{j-1}(a_l+1)=\la_1^{(i)},
$$
which is a contradiction. 

Suppose $q_1^{\la^{(i)}_1}u_i=q_1^{\la^{(j)}_{s+1}-1}q_3^{s-1}u_j$, 
where $s\in\Z_{\geq 0}$.
By a similar argument we show that it is possible only if 
$s=0,j=i-1,b_{i-1}=0$ and $\la^{(i)}_1=\la_1^{(i-1)}-1$.

Now suppose the pole  $z=q_1^{\la^{(i)}_1}u_i$ 
is canceled by poles of $\psi_{\mu^{(j)}}(u_j/z)$. We again have two cases.
For example, suppose 
$q_1^{\la^{(i)}_1}u_i=q_1^{\mu^{(j)}_{s+1}}q_3^{s}u_j$ and 
$(i,1)\neq (j,s+1)$. 
Then we necessarily have $i<j$ and
\begin{eqnarray*}
\la^{(i)}_1=\mu^{(j)}_{s+1}-\sum_{l=i}^{j-1}(a_l+1),\qquad
s=\sum_{l=i}^{j-1}(b_l+1).
\end{eqnarray*}
It implies that
$$
\mu^{(i)}_1\geq \mu^{(j)}_{1+\sum_{l=i}^{j-1}b_l}-\sum_{l=i}^{j-1}a_l>\mu^{(j)}_{\sum_{l=i}^{j-1}(b_l+1)}-\sum_{l=i}^{j-1}(a_l+1)=\la_1^{(i)}.
$$
Similarly, we obtain $\mu^{(i)}_1>\la_1^{(i)}$
in the other case of $q_1^{\la^{(i)}_1}u_i=q_1^{\mu^{(j)}_{s}-1}q_3^{s-1}u_j$.

We claim $\la^{(1)}_1=\mu^{(1)}_1$. Indeed, for $i=1$ the cancellation of the
poles with zeroes is impossible and as we saw, the cancellation of poles
$z=q_1^{\la^{(1)}_1}u_1$ and $z=q_1^{\mu^{(1)}_1}u_1$ with other poles
imply both $\mu^{(1)}_1>\la_1^{(1)}$ and $\la^{(1)}_1>\mu_1^{(1)}$,
which is a contradiction.

Next we claim that $\la^{(2)}_1=\mu^{(2)}_1$. Indeed, since the terms with $\la^{(1)}_1$ and $\mu^{(1)}_1$ cancel each other,
the cancellation of the poles $z=q_1^{\la^{(2)}_1}u_1$ and 
$z=q_1^{\mu^{(2)}_1}u_1$
with zeroes is again impossible and cancellation with other poles
leads to a contradiction.

Repeating, we obtain $\la^{(i)}_1=\mu^{(i)}_1$ for $i=1,\dots,n$. 

Cancel the corresponding factors and replace $u_i$ with $u_iq_3^{-1}$.
Then, a similar argument gives $\la^{(i)}_2=\mu^{(i)}_2$, where
$i=1,\dots,n$. Repeating the argument, we prove that the module
$\M_{\bs a,\bs b}(u)$ is tame.

Now, to prove that  $\M_{\bs a,\bs b}(u)$ is irreducible, it is sufficient to show that if vectors $\ket{\bs\la}$,
$\ket{\bs\la+\bs 1_s^{(i)}}$ are both in 
$\M_{\bs a,\bs b}(u)$, then 
$\bra{\bs\la+\bs 1_s^{(i)}}e(z)\ket{\bs\la}$ and
$\bra{\bs\la}f(z)\ket{\bs\la+\bs 1_s^{(i)}}$ are non-zero. 
It is similar to that of Proposition \ref{submodule well-defined n}.
We omit further details.
\end{proof}
The character of $\M_{\bs a,\bs b}(u)$ is given in Theorem \ref{thm:Mchar}.

The tensor action of $\E$ on the space $\F(u_1)\otimes\dots\otimes
\F(u_n)$ for generic $u_i$ does not have a limit to the case
\Ref{shift n} in the basis
$\ket{\la^{(1)},\dots,\la^{(n)}}$. This limit exists only on 
$\M_{\bs a,\bs b}(u)$. However, we think
of $\M_{\bs a,\bs b}(u)$ as ``a submodule of $\F(u_1)\otimes\dots\otimes
\F(u_n)$''.

\begin{rem}
Note that in the case of \Ref{shift n},
the action of operators
$\psi^\pm_{\pm i}$ on  $\F(u_1)\otimes\dots\otimes\F(u_n)$ is well-defined,
however the joint spectrum of $\psi^\pm_{\pm i}$ is not simple. 
For example, consider the case $n=2$, $a_1=b_1=0$. Thus we consider
$\F(u)\otimes \F(uq_2)$. Then the vectors
$\ket{\emptyset}_u\otimes \ket{(2,2)}_{uq_2}$ and 
$\ket{(1)}_u\otimes \ket{(2,1)}_{uq_2}$
have the same $\psi^\pm_{\pm i}$ eigenvalues.
\end{rem}

\subsection{Resonance in $q_1,q_3$}
Consider the tensor product of $n$ Fock modules
$\F(u_1)\otimes\dots \otimes\F(u_n)$ with $n\geq 2$.

Assume \Ref{shift n} and let $p',p$ be some integers such that
$$
 a_n= p'-1-\sum_{i=1}^{n-1}(a_i+1),\quad b_n=p-1-\sum_{i=1}^{n-1}(b_i+1)
$$
belong to $\Z_{\geq 0}$.
Assume further that 
\bea\label{resonance n}
q_1^{p'}q_3^p=1,\ \ p\neq p'.
\ena

More precisely, by 
equality \Ref{resonance n} we mean that $q_1^xq_3^y=1$ if and only if
$x=p'\kappa$, $y=p\kappa$ for some $\kappa\in\Z$.

We use a cyclic modulo $n$ convention for indices and
suffixes: $u_{n+1}=u_1$, $\la^{(0)}=\la^{(n)}$, etc. 
Let
\begin{align}
\M^{p',p}_{\bs a,\bs b}(u)
={\rm span}\{ \ket{\la^{(1)},\dots,\la^{(n)}}\ 
|\ \la^{(i)}_s\geq \la^{(i+1)}_{s+b_i}-a_i, \
{\rm where}\ s\in\Z_{\geq 1},\ i=1,\dots, n\} .
\label{Mpp}
\end{align}

The following lemma shows that the definition of $\M_{\bs a,\bs b}^{p',p}(u)$ 
is in fact a superposition of $n=2$ conditions.

\begin{lem}\label{2n} We have
$\ket{\la^{(1)},\dots,\la^{(n)}}\in
\M^{p',p}_{\bs a,\bs b}(u)$ if and only if for all $i,j$, $1\leq i<j\leq n$, 
$\ket{\la^{(i)}}\otimes\ket{\la^{(j)}}\in \M^{p',p}_{a_{ij},b_{ij}}(u_i)$. 
\end{lem}
\begin{proof}
The lemma is straightforward.
\end{proof}

We have an obvious surjective map of linear spaces:
 $\M_{\bs a,\bs b}(u)\to \M^{p',p}_{\bs a,\bs b}(u)$, sending 
$\ket{\la^{(1)},\dots,\la^{(n)}}$
to either zero or to $\ket{\la^{(1)},\dots,\la^{(n)}}$. 
In particular, the space $\M^{p',p}_{\bs a,\bs b}(u)$ inherits the
$\Z$-grading.

We define the action of operator $\psi^\pm(z), e(z),f(z)$ on
$\M^{p',p}_{\bs a,\bs b}(u)$ as the factorized action of $\E$ on
$\M_{\bs a,\bs b}(u)$. Namely, let the matrix coefficients of
operators $\psi^\pm(z), e(z),f(z)$ in the basis
$\ket{\la^{(1)},\dots,\la^{(n)}}$ be the same as the
corresponding matrix coefficients in $\M_{\bs a,\bs b}(u)$.

\begin{prop}\label{well-defined n}
The action of operators
$\psi^\pm(z), e(z),f(z)$ in
$\M^{p',p}_{\bs a,\bs b}(u)$ is well-defined 
and gives a structure of a graded $\E$-module.
\end{prop}
\begin{proof} 
Consider the case $n=2$ and set $a_1=a,b_1=b$.

It is sufficient to perform the following checks. 
\begin{enumerate}
\item 
If $\ket{\la,\mu}\in \M^{p',p}_{a,b}(u)$, then the matrix coefficients
$
\bra{\la,\mu+\one_s}\ e(z)\ \ket{\la,\mu}
$, 
$
\bra{\la-\one_s,\mu}\ f(z)\ \ket{\la, \mu}
$
are well defined. 
\item  If $\ket{\la,\mu}\in \M^{p',p}_{a,b}(u)$, then
\begin{align*}
&\ket{\la,\mu-\one_s}\not\in \M^{p',p}_{a,b}(u)\Rightarrow 
\bra{\la,\mu}\ e(z)\ \ket{\la,\mu-\one_s}=0,
\\
&\ket{\la+\one_s,\mu}\not\in \M^{p',p}_{a,b}(u)
\Rightarrow 
\bra{\la,\mu}\ f(z)\ \ket{\la+\one_s,\mu}=0\,.
\end{align*}
\end{enumerate}

All the checks are straightforward using Lemma \ref{2 point}. 

For
example, consider
$\bra{\la}\otimes\bra{\mu+\bs 1_s}\ e(z)\ \ket{\la}\otimes\ket{\mu}$. 

By
Lemma \ref{2 point} the poles of this matrix coefficient happen if
$vq_2^{1-s}/(uq_2^{1-l})=q_1^{\la_l-l-\mu_s+s}$ or
$uq_2^{1-s}/(vq_2^{1-l})=q_1^{\la_l-l-\mu_s+s-1}$. 
Which means $q_1^{\la_l-\mu_s+a+1}q_3^{l-s+b+1}=0$ or 
$q_1^{\la_l-\mu_s+a_1+1}q_3^{l-s+b+1}=0$. Equivalently, due to \Ref{resonance n},
\begin{align*}
l&=s-b-1+\gamma(b_1+b_2+2),\qquad \la_l-\mu_s+a_1+1=\gamma(a_1+a_2+2)\quad {\rm or}\\ 
l&=s-b-1+\gamma(b_1+b_2+2),\qquad  \la_l-\mu_s+a_1=\gamma(a_1+a_2+2) 
\end{align*}
for some 
$\gamma\in\Z$. 

Therefore
$\la_s=\mu_{s+b_1+1-\gamma(b_1+b_2+2)}+\gamma(a_1+a_2+2)-a_1-1$ or
$\la_s=\mu_{s+b_1+1-\gamma(b_1+b_2+2)}+\gamma(a_1+a_2+2)-a_1$.

Let, first, $\gamma<0$. Then
$$
\la_l-\mu_{l-\gamma(b_1+b_2+2)+b_1+1}\geq \la_l-\mu_{l-\gamma+b_1+1}
+\gamma(a_1+a_2)
\geq \la_l-\mu_{l+b_1}+\gamma(a_1+a_2)\geq -a_1+\gamma(a_1+a_2),
$$
and therefore the poles of the matrix coefficient do not occur. 

Let $\gamma>0$. Then
\begin{eqnarray*}
\la_l-\mu_{l-\gamma(b_1+b_2+2)+b_1+1)}\leq 
\la_{l-(\gamma-1)(b_1+b_2)}-\mu_{l-\gamma(b_1+b_2+2)+b_1+1}+(\gamma-1)(a_1+a_2)
 \hspace{25pt}\\
\leq a_2+(\gamma-1)(a_1+a_2),
\end{eqnarray*}
and again, such a pole is impossible.

Let now $\gamma=0$. That is $s=l+b_1+1$ and $\la_l-\mu_{l+b_1+1}=-a_1$
or $\la_l-\mu_{l+b_1+1}=-a_1-1$. Since $\la_l-\mu_{l+b_1+1}\geq
\la_l-\mu_{l+b_1}\geq -a_1$, the second case is impossible and in the
first case we have $\mu_{l+b_1+1}=\mu_{l+b_1}$ and therefore our
matrix coefficient was zero already for generic $u,v,q_1,q_2$.  Note that
we use here that $b_1\geq 0$, otherwise in the case $l=1$ the index
of $\mu_{s-1}=\mu_{l+b_1}$ is non-positive and the coefficient does
not have to be zero.

We omit further details.

The general case of Proposition \ref{well-defined n} 
reduces to the case of $n=2$ by Lemma \ref{2n}.
\end{proof}

\begin{thm}\label{tame thm 2}
Assume in addition that $p>n$. Then
the $\E$-module $\M_{\bs a,\bs b}^{p',p}(u)$ is an irreducible, tame, 
highest weight $\E$-module 
with highest weight $\prod_{i=1}^n\psi^\pm_{\emptyset}(u_i/z)$.
\end{thm}
\begin{proof}
The proof is similar to the proof of Theorem \ref{tame thm 1}. 

Assume that for some $\ket{\la^{(1)},\dots,\la^{(n)}}, 
\ket{\mu^{(1)},\dots,\mu^{(n)}}\in \M^{p',p}_{\bs a,\bs b}(u)$, we have
\be
\prod_{i=1}^n \psi_{\la^{(i)}}(u_i/z)=\prod_{i=1}^n \psi_{\mu^{(i)}}(u_i/z).
\en
We then show that this implies $\la^{(i)}=\mu^{(i)}$.

For example, let us check that the pole $z=q^{\la_1^{(i)}}u_i$ is not
canceled by the zero $z=q_1^{\la^{(j)}_s}q_3^{s}u_j$ of
$\psi_{\la^{(j)}}(u_j/z)$, that is
$q_1^{\la^{(i)}_1}u_i=q_1^{\la^{(j)}_s}q_3^{s}u_j$.

It is easy to see that $i=j$ is impossible.

Consider the case $j>i$. 
Then for some $\kappa\in \Z$ we have
\begin{eqnarray*}
\la^{(i)}_1=\la^{(j)}_s-\sum_{l=i}^{j-1}(a_l+1)-\kappa p',\qquad 
s=\sum_{l=i}^{j-1}(b_l+1)+\kappa p.
\end{eqnarray*}
Since $s\geq 1$, we obtain that $\kappa\geq 0$.
This is impossible since
\begin{align*}
\la^{(i)}_1\geq \la^{(i)}_{1+\kappa p}-\kappa p'\geq
\la^{(j)}_{1+\sum_{l=i}^{j-1}b_l+\kappa p}-\sum_{l=i}^{j-1}a_l-\kappa p'
>\la^{(j)}_{s}-\sum_{l=i}^{j-1}(a_l+1)-\kappa p'.
\end{align*}
Here we used that $a_l,b_l\geq 0$.

In the case $j<i$, we have
\begin{eqnarray*}
\la^{(i)}_1=\la^{(j)}_s+\sum_{l=j}^{i-1}(a_l+1)-\kappa p',\qquad 
s=-\sum_{l=j}^{i-1}(b_l+1)+\kappa p.
\end{eqnarray*}
Since $s\geq 1$, we obtain that $\kappa\geq 1$.
This is impossible since
\begin{align*}
\la^{(i)}_1\geq \la^{(i)}_{1+(\kappa-1) p}-(\kappa-1) p'>
\la^{(j)}_{1-\sum_{l=j}^{i-1}b_l+\kappa p}+
\sum_{l=j}^{i-1}(a_l+1)-\kappa p'\geq\la^{(j)}_{s}+
\sum_{l=i}^{j-1}(a_l+1)-\kappa p'.
\end{align*}

The case $q_1^{\la^{(i)}_1}u_i=q_1^{\la^{(j)}_{s+1}-1}q_3^{s-1}u_j$ is
again possible only if $j=i-1$, $s=0$ and $b_{i-1}=0$.

Note that since $p>n$, there exists $i$, such that $b_{i-1}\neq 0$.
In such a case the pole $z=q^{\la_1^{(i)}}u_i$ is not canceled with a
zero. 

We omit further details.
\end{proof}
Note that if some of $b_i$ were negative then the theorem would not hold.

The character of $\M^{p',p}_{\bs a,\bs b}(u)$ is given in Theorem \ref{thm:Wn}.

The tensor action of $\E$ on the space $\M_{\bs a,\bs b}(u)$ for
generic $q_1, q_2$ does not have a limit to the case \Ref{resonance n} in
the basis $\ket{\la^{(1)},\dots,\la^{(n)}}$.  However, we think
of $\M^{p',p}_{\bs a,\bs b}(u)$ as of ``a quotient module of $\M_{\bs
  a,\bs b}(u)$'' and even as of ``a subquotient of
$\F(u_1)\otimes\dots\otimes \F(u_n)$''.

\section{Characters}\label{char sec}
All modules considered in Section \ref{mod sec} 
are graded modules with finite-dimensional graded components. Therefore
we have well-defined formal characters which we study in this section.
\subsection{Finitized characters and recursion}
Recall that we have constructed 
a family of ${\E}$-modules $\M^{p',p}_{\ba,\bb}$. 
Here $p,p'$ are positive integers 
satisfying $p,p'\ge n$, $p'\neq p$, 
and $\ba=(a_1,\dots,a_{n-1})$, $\bb=(b_1,\dots,b_{n-1})$  
$\in\Z^{n-1}_{\ge0}$ are such that 
there exist $a_n,b_n\in\Z_{\ge0}$ satisfying 
\begin{align*}
\sum_{i=1}^n(a_i+1)=p',
\quad 
\sum_{i=1}^n(b_i+1)=p.
\end{align*}
We shall always assume that $a_n,b_n$ are determined 
from $\ba,\bb$ as above. 
Throughout this section, $\bb$ will be fixed. We shall
also assume that $p'>n$. 

The module $\M^{p',p}_{\ba,\bb}$ has a basis labeled by 
the set of $n$-tuples of partitions 
\begin{align}
&P^{p',p}_{\ba,\bb} 
:=
\{(\la^{(1)},\dots,\la^{(n)}) \mid \la^{(j)}\in\mc P^+,\ 
\la^{(i)}_j\ge\la^{(i+1)}_{j+b_i}-a_i,\ {\rm where}\
i=1,\dots,n,\ j\in\Z_{>0}
\}, 
\label{part1}
\end{align}
where $\la^{(i)}=(\la^{(i)}_j)_{j>0}$ and  
$\la^{(n+1)}=\la^{(1)}$. 
In this section, we study their characters
\begin{align}
&\chi^{p',p}_{\ba,\bb}:=
\sum_{(\la^{(1)},\dots,\la^{(n)})\in P^{p',p}_{\ba,\bb}}
q^{\sum_{i=1}^n\sum_{j=1}^\infty\la^{(i)}_j}\,.
\label{char1}
\end{align}
Our goal is to show that 
they coincide with 
the characters of modules from the $\W_n$-minimal series of $\sln$-type, 
up to an overall factor corresponding to the presence of an extra 
Heisenberg algebra (see Theorem \ref{thm:Wn} below).  

%Throughout this section, 
%we shall assume that $p'>n$, $p\ge n$. 
%As before, for a vector $\mathbf{x}=(x_1,\cdots,x_n)$ 
%we write $|\mathbf{x}|=\sum_{i=1}^nx_i$. 
%We fix $\bb\in \Z^n_{\ge0}$ such that $|\bb|=p-n$. 

As a technical tool for studying \eqref{char1}, 
let us introduce a finitized version of the characters. 
%For each $\bN\in \Z^n_{\ge0}$ and $\ba\in\Z_{\ge0}^n$
%with $|\ba|=p'-n$, 
For $\bN\in\Z^n$, define the subset 
\begin{align*}
&P^{p',p}_{\ba,\bb}[\bN] 
:=\{(\la^{(1)},\dots,\la^{(n)})\in P^{p',p}_{\ba,\bb}
\mid \la^{(i)}_{N_i+1}=0,\ {\rm where}\ i=1,\dots,n\}\,,
\end{align*}
and its character
\begin{align*}
&\chi^{p',p}_{\ba,\bb}[\bN]:=
\sum_{(\la^{(1)},\dots,\la^{(n)})\in P^{p',p}_{\ba,\bb}[\bN]}
q^{\sum_{i=1}^n\sum_{j=1}^\infty\la^{(i)}_j}\,.
\end{align*}
We set also 
\begin{align}
&\chi^{p',p}_{\ba,\bb}[\bN]=0
\quad \text{if $N_i<0$ for some $i$}.
\label{rec3}
\end{align}
Clearly we have
\begin{align}
&\chi^{p',p}_{\ba,\bb}[\bs 0]=1\,.
%\quad 
%(\ba\in \Z^n_{\ge0}).
\label{rec1}
\end{align}

In the following, we extend 
the suffix $i$ for $a_i$ by $a_{i+n}=a_i$. 
A similar convention will be used for $b_i,N_i$.

\begin{prop}
The finitized 
characters $\chi^{p',p}_{\ba,\bb}[\bN]$ satisfy 
the following recursion relations for each $i=1,\dots,n$:
\begin{align}
\chi^{p',p}_{\ba,\bb}[\bN]
&
=\chi^{p',p}_{\ba,\bb}[\bN-\one_i]
+
q^{N_i}\chi^{p',p}_{\ba-\one_{i-1}+\one_i,\bb}[\bN]
\qquad
\text{if $N_{i+1}-N_i\le b_i$ and $a_{i-1}\ge1$}\,,
\label{rec2}
\\
\chi^{p',p}_{\ba,\bb}[\bN]&=\chi^{p',p}_{\ba,\bb}[\bN-\one_i]
\qquad \text{if $N_{i}-N_{i-1}= b_{i-1}+1$ and $a_{i-1}=0$}\,.
\label{rec4}
\end{align}
In the right hand side of \eqref{rec2}, 
$\ba-\one_{i-1}+\one_i$ means 
$\ba+\one_1$ for $i=1$ and $\ba-\one_{n-1}$ for $i=n$. 
\end{prop}
\begin{proof}
We fix $i$, and assume first that $N_i>0$. 
Then the set $P^{p',p}_{\ba,\bb}[\bN]$ is partitioned into 
a disjoint union of subsets $P'\sqcup P''$, where 
\begin{align*}
&P'=\{\bla \in P^{p',p}_{\ba,\bb}[\bN]
\mid \la^{(i)}_{N_i}=0\},
\quad 
P''=\{\bla \in P^{p',p}_{\ba,\bb}[\bN] 
\mid \la^{(i)}_{N_i}>0\}\,.
\end{align*}
By the definition, 
$P'$ %is in bijective correspondence with 
coincides with $P^{p',p}_{\ba,\bb}[\bN-\one_i]$. 

Suppose $N_{i+1}-N_i\le b_i$ and $a_{i-1}\ge1$. 
For $\bla\in P''$,  the conditions 
involving $\la^{(i)}$ read 
\begin{align*}
&\la^{(i-1)}_j\ge \la^{(i)}_{j+b_{i-1}}-1-(a_{i-1}-1),
\\
&\la^{(i)}_j\ge \la^{(i+1)}_{j+b_{i}}-a_{i}.   
\end{align*}
Since $N_i+1+b_i>N_{i+1}$ and $a_i\ge0$, the second condition  
is void if $j>N_i$. Hence it can be replaced by
\begin{align*}
&\la^{(i)}_j-\theta(j\le N_i)\ge 
\la^{(i+1)}_{j+b_{i}}-(a_{i}+1),   
\end{align*}
where $\theta(P)=1$ 
if the statement $P$ is true and $\theta(P)=0$ otherwise.
This gives rise to a bijection 
$P''\to P^{p',p}_{\ba-\one_{i-1}+\one_i,\bb}[\bN]$       
sending $\bla$ to $\tilde{\bla}$, with
$\tilde{\la}^{(i')}_j=\la^{(i')}_j-\delta_{i',i}
\theta(j\le N_i)$.  
The recursion \eqref{rec2} follows from this. 
When $N_i=0$, the same consideration applies to show 
that $P^{p',p}_{\ba,\bb}[\bN]$ is in bijective correspondence
with $P^{p',p}_{\ba-\one_{i-1}+\one_i,\bb}[\bN]$.

Next suppose $N_{i}-N_{i-1}= b_{i-1}+1$ and $a_{i-1}=0$  
(in this case necessarily $N_i>0$). 
The condition 
$\la^{(i-1)}_{N_{i-1}+1}\ge \la^{(i)}_{N_i}-a_{i-1}$ 
implies $\la_{N_i}^{(i)}=0$, so that $P''=\emptyset$. 
Hence \eqref{rec4} holds true.  
\end{proof}

In general, the recursions \eqref{rec2}, \eqref{rec4} are not 
enough to determine the characters  $\chi^{p',p}_{\ba,\bb}[\bN]$ completely. 
Nevertheless they are 
in certain region of the parameters $\ba,\bs b,\bN$ 
as the following proposition shows.

\begin{prop}\label{prop:char1}
The set 
\begin{align*}
\{\chi^{p',p}_{\ba,\bb}[\bN]\mid
N_i, a_i \in \Z_{\ge0},\ 
N_{i+1}-N_i\le b_i+1,\ {\rm where}\ i=1,\dots,n\}
\end{align*}
is uniquely determined by the recursion relations \eqref{rec2},  
\eqref{rec4},  
along with the initial condition \eqref{rec1}
and the boundary condition   \eqref{rec3}.
\end{prop}
\begin{proof}
The proof is by induction on $d=\sum_{i=1}^nN_i$. 
When $d=0$, there is nothing to show. 
Suppose $d>0$, and assume that the assertion is true 
for $\sum_{i=1}^n N_i<d$. We divide into two cases,  
\begin{align*}
&\text{ Case a): $N_{i+1}-N_i\le b_i$ for all $i=1,\dots,n$},
\\
&\text{ Case b): $N_{i+1}-N_i= b_i+1$ for some $i$}.
\end{align*}

Consider Case a). Since %$\sum_{i=1}^n a_i=p'-n>0$,  
$p'-n>0$,  there is an $i$ such that $a_{i-1}>0$. 
Applying \eqref{rec2} successively for $i$, $i+1$, $\dots$,  
we obtain
\begin{align*}
\chi^{p',p}_{\ba,\bb}[\bN]
&=\chi^{p',p}_{\ba,\bb}[\bN-\one_i]
+
q^{N_i}\chi^{p',p}_{\ba-\one_{i-1}+\one_i,\bb}[\bN]
\\
&=\chi^{p',p}_{\ba,\bb}[\bN-\one_i]
+q^{N_i}\chi^{p',p}_{\ba-\one_{i-1}+\one_i,\bb}[\bN-\one_{i+1}]
+q^{N_i+N_{i+1}}\chi^{p',p}_{\ba-\one_{i-1}+\one_{i+1},\bb}[\bN]
\\
&=\dots\\
&=\sum_{j=i}^{i+n-1}
q^{N_i+\dots+N_{j-1}}
\chi^{p',p}_{\ba-\one_{i-1}+\one_{j-1},\bb}[\bN-\one_j]
+q^{N_1+\dots+N_{n}}
\chi^{p',p}_{\ba,\bb}[\bN]\,.
\end{align*}
Since $N_1+\dots+N_n=d>0$, 
$\chi^{p',p}_{\ba,\bb}[\bN]$ is determined 
in terms of those with $\sum_{i=1}^nN_i<d$. 

Next consider Case b). 
Since $\sum_{i=1}^n(b_i+1)>0$, 
we cannot have the equality 
$N_{i+1}-N_{i}= b_{i}+1$ for all $i$.   
Choose an $i$ such that 
$N_{i}-N_{i-1}= b_{i-1}+1$ and $N_{i+1}-N_{i}\le b_{i}$.  
If $a_{i-1}=0$, then 
\eqref{rec4} implies 
$\chi^{p',p}_{\ba,\bb}[\bN]=\chi^{p',p}_{\ba,\bb}[\bN-\one_i]$
and we are done. Otherwise \eqref{rec2} is applicable. 
Repeating it $a_{i-1}$ times we obtain 
\begin{align*}
\chi^{p',p}_{\ba,\bb}[\bN]
=\sum_{j=1}^{a_{i-1}}q^{(j-1)N_i}
\chi^{p',p}_{\ba-(j-1)(\one_{i-1}-\one_i),\bb}[\bN-j \one_i]
+q^{a_{i-1}N_i}
\chi^{p',p}_{\ba-a_{i-1}(\one_{i-1}-\one_i),\bb}[\bN]\,.
\end{align*}
The last term reduces to the case $a_{i-1}=0$ already 
discussed above.
\end{proof}

\subsection{Bosonic formulas and comparison to $\W_n$ characters}
Our next task is to relate $\chi^{p',p}_{\ba,\bb}$
to the characters from the $\W_n$-minimal series. 
Let us prepare some notation concerning the affine Lie algebra $\widehat{\mathfrak{sl}}_n$. 
Denote the simple roots by $\alpha_0,\dots,\alpha_{n-1}$ and 
the fundamental weights by $\omega_0,\dots,\omega_{n-1}$.  
We set $\rho=\sum_{i=0}^{n-1}\omega_i$. 
Let $W=S_n\ltimes Q$ be the affine Weyl group of type $A^{(1)}_{n-1}$, 
where 
$Q=\oplus_{i=1}^{n-1}\Z \alpha_i$ denotes the classical root lattice. 
Let further 
$L=\oplus_{i=0}^{n-1}\Z \omega_i$ be the weight lattice and 
$L^+_l=\{\sum_{i=0}^{n-1} c_i\omega_i\mid 
c_0,\dots,c_{n-1}\in\Z_{\ge0}, \ \sum_{i=0}^{n-1}c_i= l\}$ be
 the set of dominant integral weights of level $l$. 

The characters of the irreducible modules from 
the $\W_n$-minimal series of $\sln$-type are parametrized by a pair of 
dominant integral weights $(\bet,\bxi)\in L^+_{p'-n}\times L^+_{p-n}$. 
Explicitly they are given by the alternating series \cite{FKW},
\begin{align}\label{bar}
\chb^{p',p}_{\bets,\bxis}
&=\sum_{w\in W}(-1)^{\ell(w)}
q^{\frac{p'p}{2}
\Bigl|\frac{w*\bxis-\bxis}{p}\Bigr|^2
+
\bigl(\frac{w*\bxis-\bxis}{p},
p'(\bxis+\rho)-p(\bets+\rho)\bigr)}\,
\\
&=\sum_{\sigma\in S_n}
(-1)^{\ell(\sigma)}
\sum_{\alpha\in Q}
q^{\frac{p'p}{2}(\alpha,\alpha)
+(p'\sigma(\bxis+\rho)-p(\bets+\rho),\alpha)
+(\bxis+\rho-\sigma(\bxis+\rho),\bets+\rho)}
\nn\,.
\end{align}
Here $w*\bxi=w(\bxi+\rho)-\rho=\sigma(\bxi+\rho)-\rho+p\al$, where $w=(\sigma,\alpha)$, 
and $\ell(w)$ denotes the length function. 

We need also their finitization. 
For $\bN\in \Z^n_{\ge0}$ and $\bet,\bxi\in L$, 
define
\begin{align*}
\chb^{p',p}_{\bets,\bxis}[\bN]
&=\sum_{w\in W}(-1)^{\ell(w)}
q^{\frac{p'p}{2}
\Bigl|\frac{w*\bxis-\bxis}{p}\Bigr|^2
+
\bigl(\frac{w*\bxis-\bxis}{p},
p'(\bxis+\rho)-p(\bets+\rho)\bigr)
}\,
%&=\sum_{\sigma\in S_n}
%(-)^{\ell(\sigma)}
%\sum_{\alpha\in Q_n}
%q^{\frac{p'p}{2}(\alpha,\alpha)
%+(p'\sigma(\Lambda+\rho)-p(\Lambda'+\rho),\alpha)
%+(\Lambda+\rho-\sigma(\Lambda+\rho),\Lambda'+\rho)}
\\
&\qquad\qquad 
\times
(q)_{|\bN|}
\prod_{i=1}^n\frac{1}{(q)_{N_i-(w*\bxis-\bxis,\omega_i-\omega_{i-1})}}
\,. 
\end{align*}
Here $(q)_m=\prod_{i=1}^m(1-q^i)$ for $m\in\Z_{\ge0}$, $|\bN|=\sum_{i=1}^nN_i$.
We set also 
\begin{align*}
\frac{1}{(q)_m}=0\quad \text{ if $m<0$}.
\end{align*}
We retain the modulo $n$ convention for the indices, such as
$\omega_n=\omega_0$. 

\begin{prop}\label{lem-ch2}
\begin{itemize}
\item[(i)] For all $\bxi,\bet\in L$ and $i=1,\dots,n$, 
we have 
\begin{align*}
&\chb^{p',p}_{\bets,\bxis}[\bN]
=
q^{N_i}
\chb^{p',p}_{\bets-\omega_{i-1}+\omega_{i},\bxis}[\bN]
+(1-q^{|\bN|})\chb^{p',p}_{\bets,\bxis}[\bN-\one_i]\,.
\end{align*}
\item[(ii)] If $N_{i+1}=N_i+(\bxi+\rho,\alpha_i)$ 
and $(\bet+\rho,\alpha_i)=0$ for $i=1,\dots,n$, 
then
\begin{align*}
&\chb^{p',p}_{\bets,\bxis}[\bN]=0\,.
\end{align*}
\item[(iii)] If 
$\bxi\in L^+_{p-n}$, 
then 
\begin{align*}
&\chb^{p',p}_{\bets,\bxis}[\mathbf{0}]=1\,.
\end{align*}
\end{itemize}
\end{prop}
\begin{proof}
The relation (i) can be verified directly, term by term.  
To see (ii), let $\sigma_i$ be the simple reflection with respect
to the root $\alpha_i$. 
The assumption can be written as
\begin{align*}
&N_i-(w*\bxi-\bxi,\omega_i-\omega_{i-1})
=N_{i+1}-
((\sigma_iw)*\bxi-\bxi,\omega_{i+1}-\omega_{i}),
\\
&\sigma_i*\bet=\bet\,.
\end{align*}
Under these circumstances, the terms with $w$ and $\sigma_iw$ 
cancel out in the sum pairwise. 

Finally, under the assumption of (iii) and $\bN=\mathbf{0}$, 
only the term with $w=\mathrm{id}$ survives. 
\end{proof}

\begin{prop}
For all $\bN, \ba, \bb$
 such that $N_i,a_i,b_i\ge0$
and  
$N_{i+1}-N_i\le b_i+1$ for $i=1,\dots,n$, 
we have the equality
\begin{align*}
&\chi^{p',p}_{\ba,\bb}[\bN]
=
\frac{1}{(q)_{|\bN|}}\ 
%\chb^{p',p}_{\Lambda',\Lambda}[\bN]\,, 
\chb^{p',p}_{\bets,\bxis}[\bN]\,, 
\\
&\bet=\sum_{i=1}^{n}a_i\omega_i, 
\quad \bxi=\sum_{i=1}^{n}b_i\omega_i\,.
\end{align*}
We recall that $a_n=p'-n-\sum_{i=1}^{n-1}a_i$, 
$b_n=p-n-\sum_{i=1}^{n-1}b_i$. 
\end{prop}
\begin{proof}
This follows from Proposition \ref{lem-ch2} 
and Proposition \ref{prop:char1}. 
\end{proof}

Letting $N_i\to \infty$,  
we arrive at the following result. 
\begin{thm}\label{thm:Wn}
The character of the module $\M^{p',p}_{\ba,\bb}$ is given by
\begin{align*}
&\chi^{p',p}_{\ba,\bb}=\frac{1}{(q)_{\infty}}\ 
\chb^{p',p}_{\bets,\bxis}\,,
\\
&\bet=\sum_{i=1}^{n}a_i\omega_i\,,
\quad
\bxi=\sum_{i=1}^{n}b_i\omega_i\,.
\end{align*}
\end{thm}

\subsection{Characters of  $\M_{\ba,\bb}$}
The module $\M_{\ba,\bb}$ has a basis labeled by 
the set of $n$-tuples of partitions 
\begin{align*}
P_{\ba,\bb} 
:=
\{(\la^{(1)},\dots,\la^{(n)}) \mid \la^{(i)}\in\mc P,\ 
\la^{(i)}_j\ge\la^{(i+1)}_{j+b_i}-a_i,\ {\rm where}\ 
i=1,\dots,n-1,\ j\in\Z_{>0}
\}.
\end{align*}
Define their characters
\begin{align*}
\chi_{\ba,\bb}:=
\sum_{(\la^{(1)},\dots,\la^{(n)})\in P_{\ba,\bb}}
q^{\sum_{i=1}^n\sum_{j=1}^\infty\la^{(i)}_j}\,.
\end{align*}
\begin{thm}\label{thm:Mchar}
We have
$$
\chi_{\ba,\bb}=\frac{1}{(q)_\infty}
\sum_{w\in S_n} (-1)^{\ell(w)}q^{(\bs\xi+\rho-w(\bs\xi+\rho),\bs\eta+\rho)}.
$$
\end{thm}
\begin{proof}
Clearly, the set  $P_{\ba,\bb}$ is the limit of the set  $P_{\ba,\bb}^{p',p}$ as $p',p\to \infty$. The theorem then follows from Theorem \ref{thm:Wn}.
\end{proof}

%%%%%%%%%%%%%%%%%%%%%%%%%%%%%%%%%%%%%%%%%%%%%%%%%%%%%%%%%%

\newcommand{\lab}{\mbox{\boldmath$\la$}}

\section{Isomorphisms of representations}\label{iso sec}
In this section we establish several isomorphisms 
between representations of $\E$ discussed in this paper. 
All these isomorphisms preserve the basis described in terms of partitions.
This is no wonder since the modules are tame. 
In general, we expect that any two highest weight $\E$-modules with
the same highest weight are isomorphic. We check this statement here in
several cases. At the moment our proofs are strictly computational.

\subsection{Permutations of factors in the tensor products of Fock spaces}
In this section we assume that $q_1,q_2,u_1,\dots,u_n$ are generic. 

\begin{thm}\label{R-mat}
Let $\si\in S_n$ and let $q_1,q_2,u_1,\dots,u_n$ be generic.
There exist non-zero constants $b_{\bs\la}$, where 
$\bs\la=(\la^{(1)},\dots,\la^{(n)})\in {\mc P}^{n}$, such that the map
\begin{align*}
\iota:\ 
\F(u_1)\otimes \dots \otimes \F(u_n) &
\to \F(u_{\si(1)})\otimes \dots \otimes \F(u_{\si(n)}),\\
\ket{\la^{(1)}}\otimes \dots \otimes \ket{\la^{(n)}}& \mapsto 
b_{\bs \la}\ket{\la^{(\si(1))}}\otimes \dots \otimes \ket{\la^{(\si(n))}}
\end{align*}
is an isomorphism of $\E$-modules.
\end{thm}
\begin{proof}
It is sufficient to prove the theorem in the case $n=2$. 

Let $n=2$ and $\si=(12)$.
It is necessary and sufficient to show that there exist coefficients
$b_{\bs \la}$ satisfying the conditions
\begin{align*}
&b_{\bs\la+\bs 1^{(i)}_{s}}\bra{\lab+{\bs 1}^{(i)}_{s}}e(z)\ket{\lab}=
b_{\bs \la}
\bra{\bs\la'+\bs 1^{(3-i)}_{s}}e(z)\ket{\bs \la'},\\
&b_{\bs\la}\bra{\lab}f(z)\ket{\lab+{\bs 1}^{(i)}_{s}}=
b_{\bs\la+\bs 1^{(i)}_{s}}
\bra{\bs \la'}f(z)\ket{\bs\la'+\bs 1^{(3-i)}_{s}}.
\end{align*}
Here $i=1,2$ and if $\bs\la=(\la^{(1)},\la^{(2)})$ then
$\bs\la'=(\la^{(2)},\la^{(1)})$.

In order that these equations for $b_{\bs \la}$ to be consistent
the following conditions are necessary and sufficient.
\begin{align}
&\frac
{\bra{\lab}f(w)\ket{\lab+{\bs 1}^{(i)}_{s}}}
{\bra{\lab'}f(w)\ket{\lab'+{\bs 1}^{(3-i)}_{s}}}
=\frac
{\bra{\lab'+{\bs 1}^{(3-i)}_{s}}e(z)\ket{\lab'}}
{\bra{\lab+{\bs 1}^{(i)}_{s}}e(z)\ket{\lab}},\label{CON1}\\
&\frac{\bra{\lab+{\bs 1}^{(i)}_{s}+{\bf1}^{(j)}_{t}}e(z)\ket{\lab+{\bs 1}^{(i)}_{s}}}
{\bra{\lab'+\bs 1^{(3-i)}_{s}+\bs 1^{(3-j)}_{t}}e(z)\ket{\lab'+\bs 1^{(3-i)}_{s}}}
\cdot\frac{\bra{\lab+{\bs 1}^{(i)}_{s}}e(w)\ket{\lab}}
{\bra{\lab'+\bs 1^{(3-i)}_{s}}e(w)\ket{\lab'}}\label{CON2}\\
&=\frac{\bra{\lab+{\bf1}^{(i)}_{s}+{\bf1}^{(j)}_{t}}e(z)\ket{\lab+{\bs 1}^{(j)}_{t}}}
{\bra{\lab'+\bs 1^{(3-i)}_{s}+\bs 1^{(3-j)}_{t}}e(z)\ket{\lab'+\bs 1^{(3-j)}_{t}}}
\cdot\frac{\bra{\lab+{\bs 1}^{(j)}_{t}}e(w)\ket{\lab}}
{\bra{\lab'+\bs 1^{(3-j)}_{t}}e(w)\ket{\lab'}}.\nn
\end{align}
The precise meaning of such equations is as follows. 
Suppose that $\delta_i(z)=c_i\sum_{n\in\Z}(u/z)^n$, where $i=1,2$, 
are delta functions with the same support multiplied by 
non-zero constants $c_i$.  Then by the ratio $\delta_1(z)/\delta_2(z)$ we mean
the ratio $c_1/c_2$. For example, we have
\begin{align*}
\frac{\bra{\la + {\bf 1}_i}e(z)_i\ket{\la}}
{\bra{\la}f(z)_i\ket{\la + {\bf 1}_i}}=q_1^{-1}.
\end{align*}

Equations  \Ref{CON1}, \Ref{CON2} are checked by a straightforward computation.
\end{proof}

\subsection{The $Z_n$ symmetry of $\mc M^{p',p}_{\bs a,\bs b}$}
In this section we assume that the parameters $q_1,q_2,u_1,\dots,u_n$ satisfy
\Ref{shift n} and \Ref{resonance n}.

\begin{thm}
Let parameters $q_1,q_2,u_1,\dots,u_n$ satisfy
\Ref{shift n} and \Ref{resonance n}.
There exist non-zero constants $c_{\bs\la}$, where 
$\bs\la\in P^{p',p}_{\bs a,\bs b}$,
such that the map
\begin{align*}
\iota:\ 
\M^{p',p}_{(a_1,\dots,a_{n-1}),(b_1,\dots,b_{n-1})} &\to 
\M^{p',p}_{(a_2,\dots,a_n),(b_2,\dots,b_n)}, 
\\
\ket{\la^{(1)}}\otimes \dots \otimes 
\ket{\la^{(n-1)}}\otimes \ket{\la^{(n)}}& \mapsto 
c_{\bs \la}\ket{\la^{(2)}}\otimes \dots \otimes 
\ket{\la^{(n)}}\otimes \ket{\la^{(1))}}
\end{align*}
is an isomorphism of $\E$-modules.
\end{thm}
\begin{proof}
Clearly the set theoretic map 
\begin{align*}
\iota:\ 
P^{p',p}_{(a_1,\dots,a_{n-1}),(b_1,\dots,b_{n-1})} &\to 
P^{p',p}_{(a_2,\dots,a_n),(b_2,\dots,b_n)}, 
\\
(\la^{(1)},\dots, \la^{(n-1)},\la^{(n)})& \mapsto 
(\la^{(2)}, \dots, \la^{(n)},\la^{(1)})
\end{align*}
is a bijection.

Then the equations on the constants $c_{\bs\la}$ are the same as in
Theorem \ref{R-mat} with $\si=(1,2,\dots,n)$. Therefore, the theorem follows.
\end{proof}

\subsection{The modules $\mc M^{n+r,n+1}_{\bs a,\bs 0}$ and $G^{n,r}_{\bs a}$}

In this section we assume the following resonance condition
\begin{align}
q_1^{n+r}q_3^{n+1}=1.\label{RESONANCE}
\end{align}
We consider the special case of $\mc M^{p',p}_{\bs a,\bs b}$ where 
\begin{align*}
&p'=n+r,\ p=n+1,\ r=\sum_{i=1}^na_i,\\
&\ba=(a_1,\ldots,a_{n-1}),\ \bb={\bf 0}=(\underbrace{{0,\dots,0}}_{n-1}).
\end{align*}

We abbreviate the representation $\M^{n+k,n+1}_{\bf a,0}$ to 
$\M_{\bar{\bf a}}$, where $\bar{\bf a}=(a_1,\dots,a_n)$, and the set of $n$-tuple partitions 
$P^{n+r,n+1}_{\bs a,\bs 0}$ to $P_{\bar{\bs a}}$:
\begin{align*}
P_{\bar{\bs a}}=\{\lab=(\la^{(1)},\dots,\la^{(n)})\,
|\ \la^{(i)}\in\mc P^+, \ \la^{(i)}_j\geq\la^{(i+1)}_j-a_i,\ {\rm where}\ i=1,\dots,n,\ j\in\Z_{\geq 0}\}.
\end{align*}
Similarly, we abbreviate the representation $G^{n,r}_{\bs a}$ to $G_{\bar {\bs a}}$
and the set $\mc P^{n,r}_{\bs a}$ to $\mc P_{\bar {\bs a}}$.

Our aim is to prove that two representations
$\M_{\bar{\bf a}}$ and $G_{\bar{\bf a}}$ are isomorphic:
\begin{align}
\iota:\M_{\bar{\bf a}}\buildrel\simeq\over\longrightarrow G_{\bar{\bf a}}.\label{IOTA}
\end{align}
Recall that the vector space $G_{\bar{\bf a}}$ has 
a basis labeled by the set $\mc P_{\bar{\bs a}}$
of $(n,r)$-admissible partitions.
Let $\iota$ be a bijection given by
\begin{align*}
&\iota:\ P_{\bar{\bf a}}\to \mc P_{\bar{\bf a}}, \quad \bs\la\mapsto \La,\\
&\La_{ns+i}=\la^{(i)}_{s+1}+\Lambda^0_{ns+i},\ {\rm for}\ i=1,\dots,n,\ s\in\Z_{\geq 0}.
\end{align*}

\begin{thm}\label{iso thm}
There exist non-zero constants $c_\La$, $\La\in \mc P_{\bar{\bf a}}$, 
such that the linear map 
\begin{align}\label{ISO}
\iota:\ \M_{\bar{\bf a}}&\to G_{\bar{\bf a}}, \\
\ket{\bs\la} &\mapsto c_\La \ket{\La}, \qquad \La=\iota(\bs \la),
\end{align}
is an isomorphism of graded $\E$-modules.
\end{thm}
\begin{proof}
The proof is similar to the proof of Theorem \ref{R-mat}. 
However, the checks are slightly more involved and we give some details.

Note that the vectors $e(z)\ket{\Lambda}$ is a finite linear
combination of the vectors $\ket{\Lambda+\bs 1_j}$, where 
$\Lambda+\bs 1_j=
(\Lambda_1,\Lambda_2,\dots,\Lambda_{j-1},\Lambda_j+1,\Lambda_{j+1},\dots)$,
and $f(z)\ket{\Lambda}$ is a finite linear
combination of the vectors $\ket{\Lambda-\bs 1_j}$.

Therefore, it is necessary and sufficient to show that there exist coefficients
$c_\Lambda$ satisfying the conditions
\begin{align*}
&c_{\Lambda+\bs 1_{ns+i}}\bra{\lab+{\bs 1}^{(i)}_{s+1}}e(z)\ket{\lab}=
c_\Lambda
\bra{\Lambda+\bs 1_{ns+i}}e(z)\ket{\Lambda},\\
&c_\Lambda\bra{\lab}f(z)\ket{\lab+{\bs 1}^{(i)}_{s+1}}=
c_{\Lambda+\bs 1_{ns+i}}
\bra{\Lambda}f(z)\ket{\Lambda+\bs 1_{ns+i}}.
\end{align*}
Here $\bra{\lab+{\bf1}^{(i)}_{s+1}}e(z)\ket{\lab}$ and $\bra{\lab}f(z)\ket{\lab+{\bf1}^{(i)}_{s+1}}$
denote the matrix coefficients of $e(z)$ and $f(z)$ respectively in the module $\M_{\bar{\bf a}}$, and 
$\bra{\Lambda+\bs 1_{ns+i}}e(z)\ket{\Lambda}$ and $\bra{\Lambda}f(z)\ket{\Lambda+\bs 1_{ns+i}}$ are those
in the module $G_{\bar{\bf a}}$. 

In order that these equations for $c_\Lambda$ to be consistent
the following conditions are necessary and sufficient.
\begin{align}
&\frac{\bra{\lab}f(w)\ket{\lab+{\bs 1}^{(i)}_{s+1}}}
{\bra{\Lambda}f(w)\ket{\Lambda+\bs 1_{ns+i}}}
=\frac{\bra{\Lambda+\bs 1_{ns+i}}e(z)\ket{\Lambda}}
{\bra{\lab+{\bs 1}^{(i)}_{s+1}}e(z)\ket{\lab}},
\label{CONSISTENT1}\\
&\frac{\bra{\lab+{\bs 1}^{(i)}_{s+1}+{\bf1}^{(j)}_{t+1}}e(z)\ket{\lab+{\bs 1}^{(i)}_{s+1}}}
{\bra{\Lambda+\bs 1_{ns+i}+\bs 1_{nt+j}}e(z)\ket{\Lambda+1_{ns+i}}}\cdot
\frac{\bra{\lab+{\bf1}^{(i)}_{s+1}}e(w)\ket{\lab}}
{\bra{\Lambda+\bs 1_{ns+i}}e(w)\ket{\Lambda}}\label{CONSISTENT2}\\
&=\frac{\bra{\lab+{\bf1}^{(i)}_{s+1}+{\bf1}^{(j)}_{t+1}}e(z)\ket{\lab+{\bs 1}^{(j)}_{t+1}}}
{\bra{\Lambda+\bs 1_{ns+i}+\bs 1_{nt+j}}e(z)\ket{\Lambda+\bs 1_{nt+j}}}\cdot
\frac{\bra{\lab+{\bs 1}^{(j)}_{t+1}}e(w)\ket{\lab}}
{\bra{\Lambda+\bs 1_{nt+j}}e(w)\ket{\Lambda}}.\nn
\end{align}
and the condition, similar to \Ref{CONSISTENT2} for $f(z)$. Here we
follow the conventions for the ratios of delta functions with the
same support as in \Ref{CON1}, \Ref{CON2}.

We rewrite the action  of $\E$ on  $G_{\bar{\bf a}}$ 
in such a way that the matrix coefficients of $e(z)$ and
$f(z)$ acting on $G_{\bar{\bf a}}$ are given in terms of certain
quantities which are used in the expressions for the matrix
coefficients of $e(z)$ and $f(z)$ acting on $\M_{\bar{\bf a}}$. This
makes the proof of these equation shorter.

The vector in $G_{\bar{\bf a}}$ which  corresponds to 
$\Lambda\in \mc P_{\bar {\bs a}}$ is given by a semi-infinite tensor product:
\begin{align*}
\ket{\Lambda}=\ket{\La_1}_u\otimes\ket{\La_2-1}_{uq_3^{-1}}
\otimes\dots\otimes\ket{\Lambda_j-j+1}_{q_3^{-j+1}u}\otimes \dots .
\end{align*}

Let $\bra{\Lambda}\psi^\pm(z)_{ns+i}\ket{\Lambda}$ 
be the matrix coefficient
of $\psi^\pm(z)$ in the module $V(q_3^{-ns-i+1}u)$:
\begin{align*}
\psi^\pm(z)\ket{\Lambda_{ns+i}-ns-i+1}_{q_3^{-ns-i+1}u}
=\bra{\Lambda}\psi^\pm(z)_{ns+i}\ket{\Lambda}
\ket{\Lambda_{ns+i}-ns-i+1}_{q_3^{-ns-i+1}u}.
\end{align*}

Recall that we also write $\bra{\la^{(i)}}\psi^\pm(z)_{s+1}\ket{\la^{(i)}}$
for the matrix coefficient
of $\psi^\pm(z)$ in the module $V(q_2^{-s}u_i)$:
\begin{align*}
\psi^\pm(z)\ket{\la^{(i)}_{s+1}-s}_{q_2^{-s}u_i}
=\bra{\la^{(i)}}\psi^\pm(z)_{s+1}\ket{\la^{(i)}}
\ket{\la^{(i)}_{s+1}-s}_{q_2^{-s}u_i},
\end{align*}
where
\begin{align}\label{UI}
u_i=q_1^{-\sum_{j=1}^{i-1}(a_j+1)}q_3^{-i+1}u=q_1^{-c_{i-1}}q_2^{i-1}u.
\end{align}
Note that we consider these quantities as rational functions in $z$,
not as series in $z^{\pm1}$ therefore these is no distinction between
$\psi^\pm(z)$.

The following is the basic equality for the proof of the isomorphism:
\begin{align*}
q_1^{\la^{(i)}_{s+1}}q_3^su_i=q_1^{\Lambda_{ns+i}}q_2^{ns+i-1}u.
\end{align*}
From this follows
\begin{align*}
\bra{\la^{(i)}}\psi^\pm(z)_{s+1}\ket{\la^{(i)}}=
\bra{\Lambda}\psi^\pm(z)_{ns+i}\ket{\Lambda}.
\end{align*}

Using these equalities one can write the formula for the actions of 
$e(z), f(z)$ in $G_{\bar{\bf a}}$ in the form
\begin{align*}
&\bra{\Lambda+\bs 1_{ns+i}}e(z)\ket{\Lambda}
=\frac{\delta(q_1^{\la^{(i)}_{s+1}}q_3^su_i/z)}{1-q_1}
\prod_{j=1}^{i-1}
\left(\frac{1-q_2u_j/z}{1-u_j/z}\frac{1-q_3^{s+1}u_j/z}{1-q_2q_3^{s+1}u_j/z}
\prod_{m=1}^{s+1}\frac{\bra{\la^{(j)}}\psi^-(z)_m\ket{\la^{(j)}}}
{\bra{\emptyset^{(j)}}\psi^-(z)_m\ket{\emptyset^{(j)}}}\right)\\
&\times\prod_{j=i}^n\left(\frac{1-q_2u_j/z}{1-u_j/z}\frac{1-q_3^su_j/z}{1-q_2q_3^su_j/z}
\prod_{m=1}^{s}\frac{\bra{\la^{(j)}}\psi^-(z)_m\ket{\la^{(j)}}}
{\bra{\emptyset^{(j)}}\psi^-(z)_m\ket{\emptyset^{(j)}}}\right),\\
&\bra{\Lambda}f(z)\ket{\Lambda+\bs 1_{ns+i}}
=\frac{q_1\delta(q_1^{\la^{(i)}_{s+1}}q_3^su_i/z)}{1-q_1}
\prod_{j=1}^i\left(\frac{1-q_2q_3^{s+1}u_j/z}{1-q_3^{s+1}u_j/z}
\prod_{m=s+2}^\infty\frac{\bra{\la^{(j)}}\psi^+(z)_m\ket{\la^{(j)}}}
{\bra{\emptyset^{(j)}}\psi^+(z)_m\ket{\emptyset^{(j)}}}\right)\\
&\times\prod_{j=i+1}^n\left(\frac{1-q_2q_3^su_j/z}{1-q_3^su_j/z}
\prod_{m=s+1}^\infty\frac{\bra{\la^{(j)}}\psi^+(z)_m\ket{\la^{(j)}}}
{\bra{\emptyset^{(j)}}\psi^+(z)_m\ket{\emptyset^{(j)}}}\right).
\end{align*}

We also have the following formulas for the matrix coefficients 
in $\M_{\bar{\bf a}}$.
\begin{align*}
\bra{\lab+{\bs 1}^{(i)}_{s+1}}e(z)\ket{\lab}
&=\frac{\delta(q_1^{\la^{(i)}_{s+1}}q_3^su_i/z)}{1-q_1}\frac{1-q_2u_i/z}{1-u_i/z}\frac{1-q_3^su_i/z}{1-q_2q_3^su_i/z}
\prod_{m=1}^{s}\frac{\bra{\la^{(i)}}\psi^-(z)_m\ket{\la^{(i)}}}
{\bra{\emptyset^{(i)}}\psi^-(z)_m\ket{\emptyset^{(i)}}}\\
&\times\prod_{j=1}^{i-1}\left(\frac{1-q_2u_j/z}{1-u_j/z}\prod_{m=1}^\infty
\frac{\bra{\la^{(j)}}\psi^-(z)_m\ket{\la^{(j)}}}
{\bra{\emptyset^{(j)}}\psi^-(z)_m\ket{\emptyset^{(j)}}}\right),\\
\bra{\lab}f(z)\ket{\lab+{\bf1}^{(i)}_{s+1}}
&=\frac{q_1\delta(q_1^{\la^{(i)}_{s+1}}q_3^su_i/z)}{1-q_1}
\frac{1-q_2q_3^{s+1}u_i/z}{1-q_3^{s+1}u_i/z}\prod_{m=s+2}^\infty\frac
{\bra{\la^{(i)}}\psi^+(z)_m\ket{\la^{(i)}}}{\bra{\emptyset^{(i)}}\psi^+(z)_m\ket{\emptyset^{(i)}}}\\
&\times\prod_{j=i+1}^n\left(\frac{1-q_2u_j/z}{1-u_j/z}\prod_{m=1}^\infty
\frac{\bra{\la^{(j)}}\psi^+(z)_m\ket{\la^{(j)}}}
{\bra{\emptyset^{(j)}}\psi^+(z)_m\ket{\emptyset^{(j)}}}\right).
\end{align*}
Here $\emptyset^{(i)}$ is used for the trivial partition $\la^{(i)}=(0,0,\dots)$.

It is straightforward to check \eqref{CONSISTENT1} and \eqref{CONSISTENT2} 
by using these formulas.

\end{proof}

As a corollary we have a bosonic formula for the character of set of
the $(k,r)$-admissible partitions. Let
$$
\chi^{k,r}_{\bs a}=\sum_{\La\in \mc P^{k,r}_{\bs a}}q^{|\La-\La^0|}=
\sum_{\La\in \mc P^{k,r}_{\bs a}}q^{\sum_{j=1}^\infty (\La_j-\La_j^0)}
$$
be the character of set of the $(k,r)$-admissible partitions.

\begin{cor}
The character of the set of the $(k,r)$-admissible 
partitions $\P^{k,r}_{\bs a}$ 
coincides with the character of the set $P^{p',p}_{\bs a,\bs 0}$ with 
$p'=k+r$, $p=k+1$, and, in particular, we have the bosonic formula
$$
\chi^{k,r}_{\bs a}= \frac{1}{(q)_\infty}\bar\chi^{k+r,k+1}_{\bs a,\bs 0},
$$
where $\bar \chi^{k+r,k+1}_{\bs a,\bs 0}$ is given by \Ref{bar}.
\end{cor}

There are other known formulas of the sets of $(k,r)$-admissible
partitions: for bosonic formulas see \cite{FJLMM}, for fermionic
formulas in the case $(p=3)$, see \cite{FJMMT1}, \cite{FJMMT2}.

\section*{Acknowledgments}
Research of BF is partially supported by RFBR initiative
interdisciplinary project grant 09-02-12446-ofi-m, by RFBR-CNRS grant
09-02-93106, RFBR grants 08-01-00720-a, NSh-3472.2008.2 and
07-01-92214-CNRSL-a.  Research of EF was partially supported by the
Russian President Grant MK-281.2009.1, the RFBR Grants 09-01-00058,
07-02-00799 and NSh-3472.2008.2, by Pierre Deligne fund based on his
2004 Balzan prize in mathematics and by Alexander von Humboldt
Fellowship.  Research of MJ is supported by the Grant-in-Aid for
Scientific Research B-20340027.  Research of EM is supported by NSF
grant DMS-0900984. The present work has been carried out during the
visits of BF, EF and EM to Kyoto University. They wish to thank the
University for hospitality.


\begin{thebibliography}{999999999}

\bibitem[ABBBFV]{A} 
G.E. Andrews, R.J. Baxter, D.M. Bressoud, W.H. Burge, P.J. Forrester, and
G.X. Viennot, {\it Partitions with prescribed hook differences}, 
Europ. J. Comb. {\bf 8} (1987), 341-350

\bibitem[BS]{BS}
I.~Burban and O.~Schiffmann,
{\it On the Hall algebra of an elliptic curve, I},
arXiv:math/0505148 

\bibitem[C]{C} I.~Cherednik, {Double affine Hecke algebras},
Cambridge University Press, 2004

\bibitem[DI]{DI} J.~Ding and K.~Iohara, {\it Generalization of
    Drinfeld quantum affine algebras}, Lett. Math. Phys. {\bf 41}
  (1997), no. 2, 181-193

\bibitem[FFJMM]{FFJMM}
B.~Feigin, E.~Feigin, 
M.~Jimbo, T.~Miwa and E.~Mukhin,
{\it Quantum continuous $\mathfrak{gl}_\infty$:
Semi-infinite construction of representations}, arXiv:1002.3100

\bibitem[FJLMM]{FJLMM}  B. Feigin, M. Jimbo, S. Loktev, T. Miwa, E. Mukhin, 
{\it Bosonic formulas for (k,l)-admissible partitions},
Ramanujan J. {\bf 7} (2003), no. 4, 485--517


\bibitem[FJMM]{FJMM1} B.~Feigin, M.~Jimbo, T.~Miwa and E.~Mukhin,
  {\it Symmetric polynomials vanishing on the shifted diagonals and
    Macdonald polynomials}, Int. Math. Res. Not.  (2003), no. 18,
  1015--1034
  
%\bibitem[FJMM2]{FJMM2} B.~Feigin, M.~Jimbo, T.~Miwa and E.~Mukhin,
%  {\it Symmetric polynomials vanishing on the diagonals shifted by
%    roots of unity}, Int. Math. Res. Not. (2003), no. 18, 999--1014
  

\bibitem[FJMMT1]{FJMMT1}
 B. Feigin, M. Jimbo, T. Miwa, E. Mukhin, Y. Takeyama,
{\it Fermionic formulas for (k, 3)-admissible configurations}, 
Publ. Res. Inst. Math. Sci. {\bf 40} (2004), no. 1, 125--162

\bibitem[FJMMT2]{FJMMT2}
B. Feigin, M. Jimbo, T. Miwa, E. Mukhin, Y. Takeyama 
{\it Particle content of the (k,3)-configurations} 
 Publ. Res. Inst. Math. Sci. {\bf 40} (2004), no. 1, 163--220

\bibitem[FKW]{FKW} E. Frenkel, V. Kac, M. Wakimoto, {\it Characters and
    fusion rules for $W$-algebras via quantized Drinfel'd-Sokolov
    reduction}, Comm. Math. Phys. {\bf 147} (1992), no. 2, 295--328

\bibitem[FO]{FO}
B.~Feigin and 
A.~Odesskii, {\it Vector bundles on an elliptic curve and Sklyanin algebras},
Topics in quantum groups and finite-type invariants, 6584,
Amer. Math. Soc. Transl. Ser. 2, {\bf 185}, (1998)

\bibitem[FT]{FT}
B.~Feigin and A.~Tsymbaliuk,
{\it Heisenberg action in the equivariant
K-theory of Hilbert schemes via Shuffle Algebra},
arXiv:0904.1679

\bibitem[FHHSY]{FHHSY}
B.~Feigin, K.~Hashizume, A.~Hoshino, J.~Shiraishi 
and S.~Yanagida,
{\it A commutative algebra on degenerate $\mathbb{C}P^1$ and
Macdonald polynomials},
arXiv:0904.2291

\bibitem[Kac]{Kac} V.G.~Kac and A.K.~Raina, { Bombay lectures on
    highest weight representations of infinite-dimensional Lie
    algebras}, Advanced Series in Mathematical Physics, World
  Scientific Publishing Co.,1987


\bibitem[K]{K}
M.~Kasatani,
{\it Subrepresentations in the polynomial representation of the double affine Hecke algebra of type
${\rm GL}_n$ at $t^{k+1}q^{r-1}=1$}, Int. Math. Res. Not. (2005), no. 28, 1717--1742

\bibitem[M]{M}
I. Macdonald, 
Symmetric functions and Hall polynomials, 
2nd ed., Oxford University Press,
New York, 1995

%\bibitem[S]{S}
%O.~Schiffmann, {\it On the Hall algebra of an elliptic curve, II},
%arXiv:math/0508553



\bibitem[SchV1]{SV1}
O. Schiffmann and E. Vasserot,
{\it The elliptic Hall algebra, Cherednik Hecke algebras and Macdonald
polynomials}, arXiv:0802.4001

\bibitem[SchV2]{SV2}
O.~Schiffmann and E.~Vasserot,
{\it The elliptic Hall algebra and the equivariant K-theory of the Hilbert scheme of $\mathbb{A}^2$},
arXiv:0905.2555

\bibitem[SV]{SerVes} A. N. Sergeev, A. P. Veselov, {\it Generalised
    discriminants, deformed Calogero-Moser-Sutherland operators and
    super-Jack polynomials}, Adv. Math. 192 (2005), no. 2, 341--375


\end{thebibliography}
\end{document}